\documentclass{amsart}
\usepackage[utf8]{inputenc}
\pdfoutput=1
\usepackage{amssymb}
\usepackage{amsmath,amsthm,amssymb}
\usepackage{afterpage}
\usepackage{lscape}
\usepackage{pdfpages}
\usepackage{caption}

\usepackage{fancyhdr} 

\usepackage{lmodern}
\usepackage{tikz}
\newcommand{\D}{\Delta}
\definecolor{darkcyan}{rgb}{0, 0.7, 0.7}
\definecolor{darkgreen}{rgb}{0, 0.7, 0}
\definecolor{truemagenta}{rgb}{1, 0, 1}
\definecolor{amber}{rgb}{1.0, 0.75, 0.0}
\usepackage{xfp}
\usepackage{subcaption}
\usepackage{adjustbox}
\usepackage{amsthm}
\usepackage{hyperref}
\usepackage{cleveref}
\usepackage{lipsum}
\theoremstyle{plain}
\usepackage{tikz}
\usepackage{tikz-cd}
\usepackage{pgfplots}
\usepackage{quiver}
\usepackage{booktabs}
\usepgflibrary{arrows} 
\usepgflibrary[arrows] 
\usetikzlibrary{arrows} 
\usetikzlibrary[arrows]
\usepackage{longtable}
\usepackage{mathtools}
\usepackage{tabularx}
\usepackage{enumitem}
\usetikzlibrary{arrows}

\setlist{nolistsep}
\definecolor{green}{HTML}{66FF66}
\definecolor{myGreen}{HTML}{009900}

\usepackage{outlines}
\theoremstyle{definition}

\newtheorem{defn}{Definition}[section]
\newtheorem{prop}[defn]{Proposition}
\newtheorem{thm}[defn]{Theorem}

\newtheorem{example}[defn]{\textbf{Example}}
\newtheorem{lemma}[defn]{Lemma}
\newtheorem{remarkx}[defn]{Remark}
\newenvironment{remark}
  {\pushQED{\qed}\remarkx}
  {\popQED\endremarkx}
\newtheorem{notationx}[defn]{Notation}
\newenvironment{notation}
  {\pushQED{\qed}\notationx}
  {\popQED\endnotationx}

\def\bbm{\begin{pmatrix}}
\def\ebm{\end{pmatrix}}

\usepackage[toc,page]{appendix}
\newcommand{\Ext}{\text{Ext}}

\newcommand{\C}{\mathbb{C}}

\newcommand{\F}{\mathbb{F}}

\DeclareMathOperator{\im}{\mathrm{im}}

\usepackage[acronym,nomain]{glossaries}              

\definecolor{chartgray}{gray}{0.5}
\definecolor{darkcyan}{rgb}{0, 0.7, 0.7}
\definecolor{truecyan}{rgb}{0, 1, 1}
\definecolor{darkgreen}{rgb}{0, 0.65, 0}
\definecolor{truemagenta}{rgb}{1, 0, 1}
\definecolor{amber}{rgb}{1.0, 0.75, 0.0}
\definecolor{fuchsia}{rgb}{0.75, 0.0, 1.0}
\definecolor{darkcyan}{rgb}{0, 0.7, 0.7}
\definecolor{electricpurple}{rgb}{0.75, 0.0, 1.0}

\title[The algebraic Novikov Spectral Sequence for $\mathit{tmf}$]{The algebraic Novikov spectral sequence for topological modular forms}
\author{Jake Francis Baer}
\date{\today}
\subjclass[2010]{55Q10, 55T15}
\keywords{topological modular forms, motivic stable homotopy theory, Adams-Novikov spectral sequence, Adams spectral sequence, stable homotopy groups}

\pgfplotsset{compat=1.17} 
\begin{document}

\begin{abstract}
We compute the $\C$-motivic Adams spectral sequence for $\mathit{mmf}/\tau$. Up to reindexing, this spectral sequence is isomorphic to the algebraic Novikov spectral sequence for topological modular forms. We give a full analysis of the inclusion and projection maps which occur in the long exact sequence induced by multiplication by $\tau$ on Adams $E_2$-pages. Our computation is totally algebraic and can be used to obtain the $E_2$-page of the Adams-Novikov spectral sequence for topological modular forms in a way that is independent from previous computations using the elliptic curves Hopf algebroid.
\end{abstract}
 \maketitle
\setcounter{tocdepth}{1}
\tableofcontents

\section{Introduction}

A fundamental problem in algebraic topology is the determination of the stable homotopy groups of spheres. An approach which has seen great success is to consider the analogous problem in an enriched version of the stable homotopy category, or a deformation of the category of spectra \cite{stems} \cite{Rmot} \cite{C2} \cite{kerv} \cite{BHS}. In such a deformation, the homotopy groups of the unit object inherit an extra grading which spreads classes out by remembering additional information. For example, the homotopy groups of the $\C$-motivic sphere spectrum spreads out classes by providing an extra grading which remembers their Adams-Novikov filtration \cite{HKO}. 

These enriched contexts provide facets of the computation which can be approached by totally algebraic means. These techniques were flagshipped by the following two theorems.
\begin{prop}[\cite{GWX}]\label{GWXcat}
There is a subcategory $\mathrm{Mod}_{\mathbb{S}/\tau}$ of the category of $p$-complete cellular $\C$-motivic spectra which is equivalent to Hovey's stable derived category of $BP_*BP$-comodules.
\end{prop}
\begin{prop}[\cite{GWX}]\label{GWXss}
    The $\C$-motivic Adams spectral sequence for $\mathbb{S}/\tau$ is isomorphic to the algebraic Novikov spectral sequence for $BP_*$.
\end{prop}

The goal of these methods is to not only get new results, but to do so in a way that allows for the least potential indeterminacy. Such indeterminacy is characteristic of homotopical calculations. Proposition \labelcref{GWXcat} and Proposition \labelcref{GWXss} were applied in \cite{IHES} to fully compute the stable homotopy groups of spheres through the 90-stem.

Another key method in \cite{IHES} is to compare the Adams spectral sequence for the $\C$-motivic sphere spectrum to the Adams spectral sequence for the motivic modular forms spectrum $\mathit{mmf}$. The $\C$-motivic spectrum $\mathit{mmf}$ has the property that its Adams $E_2$-page is isomorphic to the cohomology of a particularly nice subalgebra of the $\C$-motivic dual Steenrod algebra, and can be thought of as the $\C$-motivic version of the classical connective topological modular forms spectrum $\mathit{tmf}$ \cite{mmf}. This identification makes the Adams spectral sequence for $\mathit{mmf}$ a sort of finite complexity approximation of the Adams spectral sequence for the sphere. Differentials and hidden extensions in the Adams spectral sequence for $\mathit{mmf}$ can then be pulled back to the Adams spectral sequence for the sphere by analyzing the map induced on Adams spectral sequences by the unit map for $\mathit{mmf}$.

Combining the above, many of the  differentials in \cite{IHES} are computed by analyzing the maps of Adams spectral sequences induced by the diagram

 \begin{equation}\label{dia}
  \centering
 \begin{tikzcd}
	{\mathbb{S}} & {\mathbb{S}/\tau} \\
	{\mathit{mmf}} & {\mathit{mmf}/\tau}
	\arrow[from=1-1, to=2-1]
	\arrow[from=1-1, to=1-2]
	\arrow[from=2-1, to=2-2]
	\arrow[from=1-2, to=2-2].
\end{tikzcd}
\end{equation}
The results in \cite{IHES} then rely on a thorough analysis of the Adams spectral sequences for $\mathbb{S}/\tau$ and $\mathit{mmf}$. In this paper we compute the $\C$-motivic Adams spectral sequence for $\mathit{mmf}/\tau$ filling in the lower right corner of \labelcref{dia}. In particular, the results of this paper were used in \cite{tmfANSS} where the remaining multiplicative uncertainties in the homotopy of $\mathit{tmf}$ were resolved.

Since the Adams spectral sequence for $\mathit{mmf}/\tau$ is isomorphic to the algebraic Novikov spectral sequence for $\mathit{tmf}$, this paper also provides a self contained and independent computation of the Adams-Novikov $E_2$-page for $\mathit{tmf}$. The $E_2$-page of the Adams-Novikov spectral sequence for $\mathit{tmf}$ was originally determined by Bauer using a series of Bockstein spectral sequences computing the cohomology of the elliptic curves Hopf algebroid \cite{bauer}. In contrast to this, our paper computes the $E_2$-page of the Adams-Novikov spectral sequence for $\mathit{tmf}$ without using elliptic curves or information about $BP_*\mathit{tmf}$.

\subsection{Outline}
In Section \labelcref{sec2} we consolidate the necessary background material on the $\C$-motivic Steenrod algebra and the $\C$-motivic May spectral sequence. In Section \labelcref{sec3} we demonstrate our techniques by computing the algebraic Novikov spectral sequence for connective real $K$-theory. In Section \labelcref{sec4} we compute the $E_2$-page $\text{Ext}_{\frac{\mathcal{A}(2)_*}{\tau}}$ of the Adams spectral sequence for $\mathit{mmf}/\tau$ as a ring using a $\C$-motivic May spectral sequence. In Section \labelcref{5} we compute the structure of $\text{Ext}_{\frac{\mathcal{A}(2)_*}{\tau}}$ as an $\text{Ext}_{\mathcal{A}(2)_*}$-module and analyze the maps 
\begin{align*}
    &i_*: \text{Ext}_{\mathcal{A}(2)_*}^{*, *, *} \to \text{Ext}_{\frac{\mathcal{A}(2)_*}{\tau}}^{*, *, *} &&q_*: \text{Ext}_{\frac{\mathcal{A}(2)_*}{\tau}}^{*, *, *} \to \text{Ext}_{\mathcal{A}(2)_*}^{*+1, *-1}
\end{align*}
which occur in the long exact sequence induced by multiplication by $\tau$ on $\text{Ext}_{\mathcal{A}(2)_*}$. In Section \labelcref{sec6} we compute the differentials in the motivic Adams spectral sequence for $\mathit{mmf}/\tau$. In Section \labelcref{sec7} we resolve hidden extensions on the $E_\infty$-page of the Adams spectral sequence for $\mathit{mmf}/\tau$ and remark on a few applications of our computation.

\subsection{Notations and Conventions}
To avoid clutter, we fix the following notations and conventions throughout the paper:
\begin{enumerate}
    \item Where useful, we have attempted to compare graded objects with explicit degrees. When considering graded objects as a whole, we use $\ast$ to denote a single grading and ${\ast, \ast}$ and  ${\ast, \ast, \ast}$ to denote a bigrading or a trigrading respectively. 
    \item All spectra and Adams spectral sequences are assumed to be $\C$-motivic and localized at 2 and $\eta$ unless explicitly stated otherwise. For example, we denote the $\C$-motivic sphere spectrum by $\mathbb{S}$ and the $\C$-motivic Eilenberg-Maclane spectrum for $\F_2$ by $H$.
    \item We will write $\mathcal{A}_*$ and $\mathcal{A}(n)_*$ for the $\C$-motivic dual Steenrod algebra $H_{\ast, \ast} H$ and its appropriate quotients.
    \item We will write $\text{Ext}_{\mathcal{A}_*}(X)$ for $\text{Ext}_{\mathcal{A}_*}(\F_2[\tau], H_{\ast, \ast} X)$ and (we also need the same convention for $\mathcal{A}(n)_*$ and $\frac{\mathcal{A}(n)_*}{\tau}$). When we omit the $X$ it is implied to be the sphere. 
\end{enumerate}

\subsection*{Acknowledgements}

We would like to thank William Balderrama, Bob Bruner, Dan Isasksen, and Andrew Salch for various helpful discussions. We would like to thank the eCHT community, particularly all the great speakers for the Fall 2022 $\mathit{tmf}$ seminar.

\section{Background}\label{sec2}

\subsection{The Motivic Steenrod Algebra}

Recall that the $\C$-motivic dual Steenrod algebra $\mathcal{A}_*$ at the prime 2 is given by
\begin{align*}
    \mathcal{A}_* \cong \frac{\F_2[\tau][\tau_0, \tau_1, ..., \xi_1, \xi_2, ...]}{\tau_i^2 + \tau \xi_{i + 1}}
\end{align*}

with coproduct
\begin{align*}
    &\Delta(\tau_i) = \tau_i \otimes 1 + \sum_{k = 0}^i \xi_{i - k}^{2^k} \otimes \tau_k &&
    \Delta(\xi_i) = \sum_{k = 0}^i \xi_{i - k}^{2^k} \otimes \xi_k.
\end{align*} and bigrading
\begin{align*}
    &|\tau| = (0, -1)&&
    |\tau_i| = (2^{i + 1} - 1, 2^i - 1)&&
    |\xi_i| = (2^{i + 1} - 2, 2^i - 1)
\end{align*}
where the first grading represents topological degree and the second grading is the motivic weight. We consider the quotient Hopf algebra $\mathcal{A}(n)_*$ which is dual to the subalgebra of the $\C$-motivic Steenrod algebra generated by $Sq^1, Sq^2, ...,Sq^{2^n}$ and is given by
\begin{align*}
    \mathcal{A}(n)_* \cong \frac{\F_2[\tau][\tau_0, \tau_1, ..., \tau_n][\xi_1, \xi_2, ..., \xi_n]}{\tau_i^2 + \tau \xi_{i+1}, \xi_1^{2^n}, \xi_{2}^{2^{n - 1}}, ..., \xi_{n}^2, \tau_n^2}.
\end{align*}
We uniformize notation by allowing $n = \infty$ and denoting $\mathcal{A}_*$ by $\mathcal{A}_*(\infty)$. We denote by $\text{Ext}_{\mathcal{A}(n)_*}$ the cohomology $\text{Ext}_{\mathcal{A}(n)_*}(\F_2[\tau], \F_2[\tau])$ of the Hopf algebra $\mathcal{A}(n)_*$. We grade $\text{Ext}^{*, *, *}_{\mathcal{A}(n)_*}$ by stem, filtration, and motivic weight respectively. In order to relate $\C$-motivic homotopy theory to the algebraic Novikov spectral sequence, we consider quotients of $\mathcal{A}(n)_*$ by the element $\tau$.

\begin{example}
    In \cite{hill} Mike Hill computes the cohomology $\text{Ext}_{\mathcal{A}(1)_*}$ of $\C$-motivic $\mathcal{A}(1)_*$. The Hopf algebra $\mathcal{A}(1)_*$ and its quotient by $\tau$ are given by
    \begin{align*}
    &\mathcal{A}(1)_* \cong \frac{\F_2[\tau][\tau_0, \tau_1, \xi_1]}{\tau_0^2 + \tau \xi_{1}, \xi_1^{2}, \tau_1^2}&&\frac{\mathcal{A}(1)_*}{\tau} \cong \frac{\F_2[\tau_0, \tau_1, \xi_1]}{\tau_0^2, \xi_1^{2}, \tau_1^2}.
\end{align*}

\end{example}
\begin{example}
    In \cite{Isa09} Dan Isaksen computes the cohomology $\text{Ext}_{\mathcal{A}(2)_*}$ of $\C$-motivic $\mathcal{A}(2)_*$. The Hopf algebra $\mathcal{A}(2)_*$ and its quotient by $\tau$ are given by
     \begin{align*}
    &\mathcal{A}(2)_* \cong \frac{\F_2[\tau][\tau_0, \tau_1, \tau_2, \xi_1, \xi_2]}{\tau_0^2 + \tau \xi_{1},\tau_1^2 + \tau \xi_{2}, \xi_1^{4}, \xi_{2}^{2},\tau_2^2} && \frac{\mathcal{A}(2)_*}{\tau} \cong \frac{\F_2[\tau_0, \tau_1, \tau_2, \xi_1, \xi_2]}{\tau_0^2, \tau_1^2, \xi_1^{4}, \xi_{2}^{2},\tau_2^2}.
\end{align*} 

The generators for $\text{Ext}_{\mathcal{A}(2)_*}$ as a ring are given in Table \ref{1}.

\end{example}

Multiplication by $\tau$ gives us a short exact sequence

\begin{align*}
    0 \to \mathcal{A}(n)_* \xrightarrow{\tau} \mathcal{A}(n)_* \to \frac{\mathcal{A}(n)_*}{\tau} \to 0
\end{align*}
which induces the long exact sequence 

\begin{center}
\begin{equation}\label{les}
\begin{tikzcd}
  \cdots \rar & \text{Ext}_{\mathcal{A}(n)_*}^{*, *, * + 1} \rar{\tau}
             \ar[draw=none]{d}[name=X, anchor=center]{}
    & \text{Ext}_{\mathcal{A}(n)_*}^{*, *, *} \rar{i_*} & \text{Ext}_{\frac{\mathcal{A}(n)_*}{\tau}}^{*, *, *} \ar[rounded corners,
            to path={ -- ([xshift=2ex]\tikztostart.east)
                      |- (X.center) \tikztonodes
                      -| ([xshift=-2ex]\tikztotarget.west)
                      -- (\tikztotarget)}]{dll}[at end]{q_*} \\      
  &\text{Ext}_{\mathcal{A}(n)_*}^{* - 1, * + 1, * + 1} \rar{\tau}
    & \cdots 
\end{tikzcd}
\end{equation}
\end{center}
on $\text{Ext}$ groups.

\begin{lemma}\label{iring}
    Let $0 \leq n \leq \infty$. The map $i_*: \text{Ext}_{\mathcal{A}(n)_*}^{*, *, *} \to \text{Ext}_{\frac{\mathcal{A}(n)_*}{\tau}}^{*, *, *}$ is a ring map.
\end{lemma}
\begin{proof}
    The map $i_*: \mathcal{A}(n)_* \to \frac{\mathcal{A}(n)_*}{\tau}$ is a map of Hopf algebras so it commutes with coproducts. This means that $i_*$ induces a map of cochain complexes for the cobar complexes for $\mathcal{A}(n)_*$ and $\frac{\mathcal{A}(n)_*}{\tau}$ respectively. Futhermore the induced map on cobar complexes preserves the concatenation product so the map induced on cohomology is a ring map.
\end{proof}

\begin{remark}
    The map $q_*$ is an $\text{Ext}_{\mathcal{A}(n)_*}$-module map. This follows from the fact that $q_*$ is the connecting morphism in the long exact sequence of $\text{Ext}$ groups induced by a short exact sequence of $\mathcal{A}(n)_*$-modules. 
\end{remark}

We use the following notation consistent with the notation for the algebraic Novikov spectral sequence for the sphere in \cite{aNSS}. 
\begin{notation}\label{barnotation}
Let $x$ be in $\text{Ext}_{\mathcal{A}(n)_*}^{*, *, *}$. We denote an element of $\text{Ext}_{\frac{\mathcal{A}(n)_*}{\tau}}^{*, *, *}$ by $x$ if $i_*(x) = x$. We denote an element of $\text{Ext}_{\frac{\mathcal{A}(n)_*}{\tau}}^{*+1, * - 1, * - 1}$ by $\overline{x}$ if $q_*(\overline{x}) = x$.
\end{notation}

\begin{remark}
 The element $\overline{x}$ is well defined up to elements in the image of $i_*$ because $\im{i_*} = \ker{q_*}$. If $a$ and $b$ are elements of  $\text{Ext}_{\mathcal{A}(n)_*}$ which are annihilated by $\tau$, then it is not necessarily true that $a \cdot \overline{b} = \overline{a} \cdot b$ in $\text{Ext}_{\frac{\mathcal{A}(n)_*}{\tau}}$. This will only be true up to an element in the image of $i_*$. 
\end{remark}

\subsection{Motivic Spectra and $\tau$}

Let $X$ be a $2$-complete cellular $\C$-motivic spectrum such that $H_{\ast, \ast} X$ is $\tau$-torsion-free so $\tau$ induces a short exact sequence 
\begin{align}\label{ses}
    0 \to H_{\ast, \ast}X \xrightarrow{i_*} H_{\ast, \ast}(X/\tau) \xrightarrow{q_*} H_{\ast - 1, \ast + 1}(X) \to 0
\end{align}
on homology. 
\begin{notation}
    We denote by $\text{Ext}_{\mathcal{A}_*}(X)$ the $E_2$-page $\text{Ext}_{\mathcal{A}_*}(\F_2[\tau], H_{\ast, \ast} X)$ of the $\C$-motivic Adams spectral sequence for $X$.
\end{notation} 
The short exact sequence \labelcref{ses} gives us a sequence 

\begin{equation}\label{cc 1}
\begin{tikzcd}
	\cdots & {\text{Ext}_{\mathcal{A}_*}^{s, f, w}(X)} & {\text{Ext}_{\mathcal{A}_*}^{s, f, w}(X/\tau)} & {\text{Ext}_{\mathcal{A}*}^{s - 1, f + 1, w + 1}(X)} & \cdots \\
	\cdots & {\pi_{s, w}X} & {\pi_{s, w}(X/\tau)} & {\pi_{s - 1, w + 1}X} & \cdots
	\arrow["{i_*}", from=1-2, to=1-3]
	\arrow["{q_*}", from=1-3, to=1-4]
	\arrow[from=2-2, to=2-3]
	\arrow[Rightarrow, from=1-2, to=2-2]
	\arrow[Rightarrow, from=1-3, to=2-3]
	\arrow[from=2-3, to=2-4]
	\arrow[Rightarrow, from=1-4, to=2-4]
	\arrow[from=2-1, to=2-2]
	\arrow[from=1-1, to=1-2]
	\arrow[from=1-4, to=1-5]
	\arrow[from=2-4, to=2-5]
\end{tikzcd}
\end{equation}
of Adams spectral sequences in which the rows are exact. 

If $X$ is a ring, the maps $i_*$ and $q_*$ give us a pair of $\text{Ext}_{\mathcal{A}_*}(X)$-module maps 
\begin{align}\label{maps}
    &i_*: \text{Ext}_{\mathcal{A}_*}^{*, *, *}(X) \to \text{Ext}_{\mathcal{A}_*}^{*, *, *}(X/\tau) && q_*: \text{Ext}_{\mathcal{A}_*}^{*, *, *}(X/\tau) \to \text{Ext}_{\mathcal{A}_*}^{*-1, *+1, *+1}(X).
\end{align}
\begin{example}
In \cite{ko} a construction is given for a $\C$-motivic ring spectrum $\mathit{kq}$ whose Betti realization is the classical connective real K-theory spectrum $\mathit{ko}$ and has the property $H_{\ast, \ast} \mathit{kq} \cong \mathcal{A}_* \Box_{\mathcal{A}(1)_*} \F_2[\tau]$ \cite{ko}.    
\end{example}

 \begin{prop}
     The $E_2$-page of the Adams spectral sequence for $\mathit{kq}/\tau$ is isomorphic to $\text{Ext}_{\frac{\mathcal{A}(1)_*}{\tau}}$.
 \end{prop}
 \begin{proof}
    Since $H_{\ast, \ast} \mathit{kq}$ is $\tau$-free, the map $\tau$ induces a short exact sequence
\begin{align*}
    0 \to H_{\ast, \ast} \mathit{kq} \to H_{\ast, \ast} \mathit{kq} \to H_{\ast, \ast}(\mathit{kq}/\tau) \to 0
\end{align*}
that gives us $H_{\ast, \ast}(\mathit{kq}/\tau) \cong \mathcal{A}_* \Box_{\frac{\mathcal{A}(1)_*}{\tau}} \F_2[\tau]$.
A change of rings theorem then gives
\begin{align*}
    \text{Ext}_{\mathcal{A}_*}(\mathit{kq}/\tau) \cong \text{Ext}_{\frac{\mathcal{A}(1)_*}{\tau}}
\end{align*}
which is the $E_2$ page of the Adams spectral sequence for $\mathit{kq}/\tau$.
 \end{proof}

\begin{example}
    In \cite{mmf} a construction is given for a $\C$-motivic ring spectrum $\mathit{mmf}$ whose Betti realization is the classical connective topological modular forms spectrum $\mathit{tmf}$ and has the property $H_{\ast, \ast} \mathit{mmf} \cong \mathcal{A}_* \Box_{\mathcal{A}(2)_*} \F_2[\tau]$.
\end{example}
 \begin{prop}
     The $E_2$-page of the Adams spectral sequence for $\mathit{mmf}/\tau$ is isomorphic to $\text{Ext}_{\frac{\mathcal{A}(2)_*}{\tau}}$.
 \end{prop}
 \begin{proof}
    Since $H_{\ast, \ast} \mathit{mmf}$ is $\tau$-free, the map $\tau$ induces a short exact sequence
\begin{align*}
    0 \to H_{\ast, \ast} \mathit{mmf} \to H_{\ast, \ast} \mathit{mmf} \to H_{\ast, \ast}(\mathit{mmf}/\tau) \to 0
\end{align*}
that gives us $H_{\ast, \ast}(\mathit{mmf}/\tau) \cong \mathcal{A}_* \Box_{\frac{\mathcal{A}(2)_*}{\tau}} \F_2[\tau]$.
A change of rings theorem then gives
\begin{align*}
    \text{Ext}_{\mathcal{A}_*}(\mathit{mmf}/\tau) \cong \text{Ext}_{\frac{\mathcal{A}(2)_*}{\tau}}
\end{align*}
which is the $E_2$ page of the Adams spectral sequence for $\mathit{mmf}/\tau$.
 \end{proof}

\subsection{May Spectral Sequence}\label{MSS}
Much of our work depends on comparing related computations via maps of spectral sequences. The May filtration is given by filtering the cobar complexes for $\mathcal{A}(n)_*$ and $\frac{\mathcal{A}(n)_*}{\tau}$ by their respective augmentation ideals. The associated spectral sequences are the May spectral sequences computing $\text{Ext}_{\mathcal{A}(n)_*}$ and $\text{Ext}_{\frac{\mathcal{A}(n)_*}{\tau}}$ respectively. 
\begin{prop}\label{2.2}
     Let $0 \leq n \leq \infty$. The reduction modulo $\tau$ map $\mathcal{A}(n)_* \to \frac{\mathcal{A}(n)_*}{\tau}$ induces a map of May spectral sequences.
\end{prop}

\begin{proof}
    The reduction modulo $\tau$ map $\mathcal{A}(n)_* \to \frac{\mathcal{A}(n)_*}{\tau}$ is a map of Hopf algebras so it commutes with coproducts. This means that the reduction modulo $\tau$ map induces a map of cochain complexes for the cobar complexes for $\mathcal{A}(n)_*$ and $\frac{\mathcal{A}(n)_*}{\tau}$ respectively. Furthermore this map preserves augmentation ideals so it preserves the May filtration on $\mathcal{A}(n)_*$ and $\frac{\mathcal{A}(n)_*}{\tau}$ respectively. We conclude that the reduction modulo $\tau$ map induces a filtration-preserving map of filtered cochain complexes and therefore induces a map of May spectral sequences that commutes with differentials. 
\end{proof}
Let $\mathcal{A}(n)_*^{cl}$ denote the quotient of the classical dual Steenrod algebra which is dual to the subalgebra generated by $Sq^1, Sq^2, ..., Sq^{2^n}$. We have 
\begin{align}
    \mathcal{A}(n)_*[\tau^{-1}] \cong \mathcal{A}(n)_*^{cl}[\tau, \tau^{-1}]
\end{align}
so $\tau$-localization gives us a comparison between the $\C$-motivic dual Steenrod algebra and the classical dual Steenrod algebra.
\begin{prop}\label{2.3}
    The $\tau$-localization map $\mathcal{A}(n)_* \to \mathcal{A}(n)_*^{cl}[\tau, \tau^{-1}]$ induces a map of May spectral sequences.
\end{prop}
\begin{proof}
    The $\tau$-localization map $\mathcal{A}(n)_* \to \mathcal{A}(n)_*^{cl}[\tau, \tau^{-1}]$ is a map of Hopf algebras so it commutes with coproducts. This means that the $\tau$-localization map induces a map of cochain complexes for the cobar complexes for $\mathcal{A}(n)_*$ and $\mathcal{A}(n)_*^{cl}[\tau, \tau^{-1}]$ respectively. Furthermore this map preserves augmentation ideals so it preserves the May filtration on $\mathcal{A}(n)_*$ and $\mathcal{A}(n)_*^{cl}[\tau, \tau^{-1}]$ respectively. We conclude that the $\tau$-localization map induces a filtration-preserving map of filtered cochain complexes and therefore induces a map of May spectral sequences that commutes with differentials. 
\end{proof}

\section{Computations over \texorpdfstring{$\mathcal{A}(1)_*$}{Lg}}\label{sec3}
We demonstrate our techniques by computing the algebraic Novikov spectral sequence for connective real $K$-theory. We compute the May spectral sequence for $\frac{\mathcal{A}(1)_*}{\tau}$ and analyze the inclusion and projection maps $i_*$ and $q_*$ in \labelcref{les}. The input data for this computation is knowledge of the May spectral sequence for $\mathcal{A}(1)_*$. The cohomology $\text{Ext}_{\mathcal{A}(1)_*}$ of $\mathcal{A}(1)_*$ was computed in \cite{hill} and \cite{ko}. For completeness, we compute $\text{Ext}_{\mathcal{A}(1)_*}$ using a May spectral sequence.

\begin{prop}[\cite{hill}, \cite{ko}]\label{3.1}
    The generators and relations for $\text{Ext}_{\mathcal{A}(1)_*}$ are given in Tables \ref{a} and \ref{b}.
    \end{prop}

\begin{proof}
Filtering the cobar complex for $\mathcal{A}(1)_*$ by powers of its augmentation ideal gives us a May spectral sequence with $E_1$-page $\F_2[h_0, h_1, h_{20}]$. By Proposition \labelcref{2.3} we have that $\tau$-localization induces a map of May spectral sequences. Therefore the $d_1$-differentials in the May spectral sequence for $\C$-motivic $\mathcal{A}(1)_*$ must recover the classical May $d_1$-differentials for $\mathcal{A}(1)_*$ upon $\tau$-localization. For example we have $d_1(h_{20}) = h_0 h_1$ in the May spectral sequence for $\C$-motivic $\mathcal{A}(1)_*$ since there is an analogous classical May differential for $\mathcal{A}(1)_*$ \cite[Example 3.2.7]{Rav86}. All $d_1$-differentials are determined by $d_1(h_{20})$ via the Leibniz rule. 

The $d_2$-differentials are compatible with the $d_2$-differentials in the May spectral sequence for classical $\mathcal{A}(1)_*$ and must preserve motivic weight. We have a classical May differential $d_2(h_{20}^2) = h_1^3$. However the class $h_1^3$ has motivic weight 3 and $h_{20}^2$ has motivic weight 2. In order to balance motivic weights we must have $d_2(h_{20}^2) = \tau h_1^3$ in the May spectral sequence for $\C$-motivic $\mathcal{A}(1)_*$. 

The May spectral sequence collapses at $E_3$ and there are no hidden extensions.
\end{proof}

\begin{prop}
    The trigraded ring $\text{Ext}_{\frac{\mathcal{A}(1)_*}{\tau}}$ is given by
    \begin{align*}
        \text{Ext}_{\frac{\mathcal{A}(1)_*}{\tau}} \cong \frac{\F_2[h_0, h_1, \overline{h_1^3}]}{h_0 h_1}
    \end{align*}
    with 
     \begin{align*}
        &|h_0| = (0, 1, 0) && |h_1| = (1, 1, 1) &&|\overline{h_1^3}| = (4, 2, 2).
    \end{align*} 
\end{prop}
\begin{remark}
    Our notation for the element $\overline{h_1^3}$ is explained by Notation \labelcref{barnotation} and Proposition \labelcref{Kindecomps}.
\end{remark}
\begin{proof}
Filtering the cobar complex for $\frac{\mathcal{A}(1)_*}{\tau}$ by powers of its augmentation ideal gives us a May spectral sequence with $E_1$-page $\F_2[h_0, h_1, h_{20}]$. The $d_1$-differentials must be compatible with the $d_1$-differentials in Proposition \labelcref{3.1} since the reduction modulo $\tau$ map induces a map of May spectral sequences by Proposition \labelcref{2.2}. Since the differential $d_1(h_{20}) = h_0 h_1$ in the proof of Proposition \labelcref{3.1} does not involve multiplication by $\tau$, it has nontrivial image under the map of May spectral sequences induced by reduction modulo $\tau$. We then have a differential $d_1(h_{20}) = h_0 h_1$ in the May spectral sequence for $\frac{\mathcal{A}(1)_*}{\tau}$.  Likewise the $d_2$-differentials must be compatible with the $d_2$-differentials in Proposition \labelcref{3.1}. For example we have $d_2(h_{20}^2) = \tau h_1^3$ in Proposition \labelcref{3.1} so reduction modulo $\tau$ gives us  $d_2(h_{20}^2) = 0$ in the May spectral sequence for $\frac{\mathcal{A}(1)_*}{\tau}$. Therefore the May spectral sequence for $\frac{\mathcal{A}(1)_*}{\tau}$ collapses at $E_2$ and there are no hidden extensions.
\end{proof}

\begin{prop}\label{3.3}
    The values of the map $i_*: \text{Ext}_{\mathcal{A}(1)_*}^{\ast, \ast, \ast} \to \text{Ext}_{\frac{\mathcal{A}(1)_*}{\tau}}^{\ast, \ast, \ast}$ on the ring generators of $\text{Ext}_{\mathcal{A}(1)_*}$ are given in Table \ref{a}.
\end{prop}
\begin{proof}
The values of $i_*$ are computed by inspection using \labelcref{les}. Specifically the kernel of $i_*$ consists of the $\tau$-divisible elements of $\text{Ext}_{\mathcal{A}(1)_*}$. The ring generators of $\text{Ext}_{\mathcal{A}(1)_*}$ are not $\tau$-divisible so they take nontrivial values under $i_*$. The particular values of $i_*$ on these generators are forced for degree reasons. For example, we have that $a$ has degree $(4, 3, 2)$ and the only element of $\text{Ext}_{\frac{\mathcal{A}(1)_*}{\tau}}$ with the same degree is $h_0 \mathord{\cdot} \overline{h_1^3}$ so we must have $i_*(a) = h_0 \mathord{\cdot} \overline{h_1^3}$. Similarly the degree of $P$ is $(8, 4, 4)$ and the only element of $\text{Ext}_{\frac{\mathcal{A}(1)_*}{\tau}}$ with the same degree is $\overline{h_1^3} \mathord{\cdot} \overline{h_1^3}$.
\end{proof}
 By Lemma \labelcref{iring} we have that $i_*$ is a ring homomorphism so Proposition \labelcref{3.3} totally determines the behavior of $i_*$ on $\text{Ext}_{\mathcal{A}(1)_*}$.

\begin{prop}
    The indecomposables of $\text{Ext}_{\frac{\mathcal{A}(1)_*}{\tau}}^{\ast, \ast, \ast}$ as an $\text{Ext}_{\mathcal{A}(1)_*}$-module are $1$ and $\overline{h_1^3}$.
\end{prop}
\begin{proof}
    Observe that the map $i_*$ equips $\text{Ext}_{\frac{\mathcal{A}(1)_*}{\tau}}$ with an $\text{Ext}_{\mathcal{A}(1)_*}$-module structure such that $1$ is the only $\text{Ext}_{\mathcal{A}(1)_*}$-module indecomposable in the image of $i_*$. By Proposition \labelcref{3.3} we have that the only elements of $\text{Ext}_{\frac{\mathcal{A}(1)_*}{\tau}}$ that are not in the image of $i_*$ are of the form $h_1^\ell \cdot \overline{h_1^3}^{2k + 1}$. Since $h_1$ and $\overline{h_1^3}^2$ are in the image of $i_*$ the only $\text{Ext}_{\mathcal{A}(1)_*}$-module indecomposables of $\text{Ext}_{\frac{\mathcal{A}(1)_*}{\tau}}$ are $1$ and $\overline{h_1^3}$. 
\end{proof}

\begin{prop}\label{Kindecomps}
    The values of the map $q_*: \text{Ext}_{\frac{\mathcal{A}(1)_*}{\tau}}^{\ast, \ast, \ast} \to \text{Ext}_{\mathcal{A}(1)_*}^{*-1, *+1, *+1}$ on the $\text{Ext}_{\mathcal{A}(1)_*}$-module indecomposables of $\text{Ext}_{\frac{\mathcal{A}(1)_*}{\tau}}$ are given by
    \begin{align*}
        &q_*(1) = 0 && q_*(\overline{h_1^3}) = h_1^3.
    \end{align*}
\end{prop}
\begin{proof}
The values of $q_*$ are computed by inspecting the long exact sequence \labelcref{les}. Specifically $q_*$ surjects onto the $\tau$-torsion in $\text{Ext}_{\mathcal{A}(1)_*}$. Since $h_1^3$ is $\tau$-torsion, it is in the image of $q_*$. The only possibility is that $q_*(\overline{h_1^3}) = h_1^3$.
\end{proof}
Since $q_*$ is an $\text{Ext}_{\mathcal{A}(1)_*}$-module map Proposition \labelcref{Kindecomps} completely determines the behavior of $q_*$ on $\text{Ext}_{\frac{\mathcal{A}(1)_*}{\tau}}$. 

Figure \labelcref{kochart} displays $\text{Ext}_{\mathcal{A}(1)_*}$. The chart in Figure \labelcref{kochart} is graded so that the horizontal axis corresponds to the stem and the vertical axis corresponds to Adams filtration. Vertical lines denote multiplication by $h_0$ and diagonal lines of slope $1$ denote multiplication by $h_1$. A vertical arrow denotes an infinite sequence of multiplication by $h_0$. An arrow of slope 1 denotes an infinite sequence of multiplication by $h_1$. A class in Figure \labelcref{kochart} is colored gray if it is $\tau$-periodic and is colored red if it is $\tau$-torsion. 
\begin{remark}
Recall that $\text{Ext}_{\mathcal{A}(1)_*}$ is isomorphic to the $E_2$-page of the Adams spectral sequence for $\mathit{kq}$. There is no room for Adams differentials in the Adams spectral sequence for $\mathit{kq}$ or for hidden extensions on the $E_\infty$-page.
\end{remark}
Figure \labelcref{komtchart} displays $\text{Ext}_{\frac{\mathcal{A}(1)_*}{\tau}}$. As with Figure \labelcref{kochart}, the horizontal axis corresponds to stem and the vertical axis corresponds to Adams filtration. A class in Figure \labelcref{komtchart} is colored gray if it is in the image of $i_*$ and is colored red if it takes a nontrivial value under $q_*$. An extension is colored green if it is an extension from a class with nontrivial value under $q_*$ to a class in the image of $i_*$. Since $i_*(P) = \overline{h_1^3}^2$, we denote the element $\overline{h_1^3}^2$ of $\text{Ext}_{\frac{\mathcal{A}(1)_*}{\tau}}$ by $P$ as in Notation \labelcref{barnotation}.
\begin{remark}
    Recall that $\text{Ext}_{\frac{\mathcal{A}(1)_*}{\tau}}$ is isomorphic to the $E_2$-page of the Adams spectral sequence for $\mathit{kq}/\tau$. There is no room for Adams differentials in the Adams spectral sequence for $\mathit{kq}/\tau$ or for hidden extensions on the $E_\infty$-page.
\end{remark}

  \begin{table}[hbt!]
\centering
\caption{$\text{Ext}_{\mathcal{A}(1)_*}$ generators}
\begin{tabular}{l l l l} 
 \toprule
 $(s, f, w)$ & gen. & May & $i_*(-)$ \\ [0.5ex] 
 \hline
  (0, 1, 0) & $h_0$ & $h_0$& $h_0$   \\ 
 (1, 1, 1) & $h_1$ & $h_1$& $h_1$ \\ 
 (4, 3, 2) & $a$ & $h_0 h_{20}^2$ & $h_0 \mathord{\cdot} \overline{h_1^3}$\\
 (8, 4, 4) & $P$ & $h_{20}^4$& $\overline{h_1^3}^2$\\
 \bottomrule
\end{tabular}
\label{a}
\end{table}

\begin{table}[hbt!]
\centering
\caption{$\text{Ext}_{\mathcal{A}(1)_*}$ relations}
\begin{tabular}{l l} 
 \toprule
  $(s, f, w)$ & relation\\ [0.5ex] 
 \hline
   (1, 2, 1) & $h_0 h_1$\\ (3, 3, 2) & $\tau h_1^3$\\ (5, 4, 3) & $h_1 a$ \\ (8, 6, 4) & $a^2 \mathord{+} h_0^2 P$ \\\bottomrule
\end{tabular}
\label{b}
\end{table}
\newpage
\begin{figure}
		\begin{tikzpicture}
		[scale=0.83,
		>=stealth,
		label distance=0,
		label position=below,
		every label/.style={
			inner sep=0,
			scale=0.60},
		every path/.style={thick},
		dot/.style={circle,
			inner sep=0,
			minimum size=0.09cm},
		tau0/.style={
			draw={gray},
			fill={gray}},
		tau1/.style={
			draw={red},
			fill={red}},
		tau2/.style={
			draw={darkgreen},
			fill={darkgreen}},
		tau3/.style={
			draw={amber},
			fill={amber}},
		tau4/.style={
			draw={purple},
			fill={purple}},
		tauextn/.style={draw={red}},
		tau2extn/.style={draw={orange}},
		tau3extn/.style={draw={darkgreen}},
		tau4extn/.style={draw={orange}},
		tau5extn/.style={draw={orange}},
		tau6extn/.style={draw={orange}},
		cl/.style={->},
		mot/.style={draw={darkcyan}},
		tower/.style={->},
            tsq/.style={rectangle, 
			inner sep=0,
			minimum size=0.10cm,
			scale=1.0},
		h0/.style={draw={gray}},
		h1/.style={draw={gray}},
		h2/.style={draw={gray}},
		d1/.style={draw={blue}},
		d2/.style={draw={darkcyan}},
		d3/.style={draw={red}},
		d4/.style={draw={darkgreen}},
		d5/.style={draw={blue}},
		d6/.style={draw={orange}},
		d7/.style={draw={orange}},
		d9/.style={draw={orange}},
		h0tower/.style={draw={gray}, tower},
		h1tower/.style={draw={gray}, tower},
        th0tower/.style={draw={red}, tower},
        th1tower/.style={draw={red}, tower},
        t2h0tower/.style={draw={darkgreen}, tower}, 
        t2h1tower/.style={draw={darkgreen}, tower},
        t3h0tower/.style={draw={amber}, tower}, 
        t3h1tower/.style={draw={amber}, tower}]
			
		\draw[h0, tau0] (0.0,0) -- (0.0,1);
		\draw[h1, tau0] (0.0,0) -- (1.0,1);
		\node[dot,tau0, label=$ $] at (0.0,0) {};
		\draw[h0tower] (0.0,1) -- +(0.0,0.7);
		\node[dot, tau0, label=135:$ h_0$] at (0.0,1) {};
		\draw[h1, tau0] (1.0,1) -- (2.0,2);
		\node[dot, tau0, label=135:$ h_1$] at (1.0,1) {};
		\draw[h1, tau1] (2.0,2) -- (3.0,3);
		\node[dot, tau0, label=$ $] at (2.0,2) {};
		\draw[th1tower] (3.0,3) -- +(0.7,0.7);
		\node[dot, tau1, label=$ $] at (3.0,3) {};
		\draw[h0tower] (4.0,3) -- +(0.0,0.7);
		\node[dot, tau0, label=$ a$] at (4.0,3) {};
		\draw[h0, tau0] (8.0,4) -- (8.0,5);
		\draw[h1, tau0] (8.0,4) -- (9.0,5);
		\node[dot, tau0, label=$ P$] at (8.0,4) {};
		\draw[h0tower] (8.0,5) -- +(0.0,0.7);
		\node[dot, tau0] at (8.0,5) {};
		\draw[h1, tau0] (9.0,5) -- (10.0,6);
		\node[dot, tau0] at (9.0,5) {};
		\draw[h1, tau1] (10.0,6) -- (11.0,7);
		\node[dot, tau0, label=$ $] at (10.0,6) {};
		\draw[th1tower] (11.0,7) -- +(0.7,0.7);
		\node[dot, tau1, label=$ $] at (11.0,7) {};
		\draw[h0tower] (12.0,7) -- +(0.0,0.7);
		\node[dot, tau0, label=$ P a$] at (12.0,7) {};
	\end{tikzpicture}
\caption{$\text{Ext}_{\mathcal{A}(1)_*}$\label{kochart}}
\end{figure}
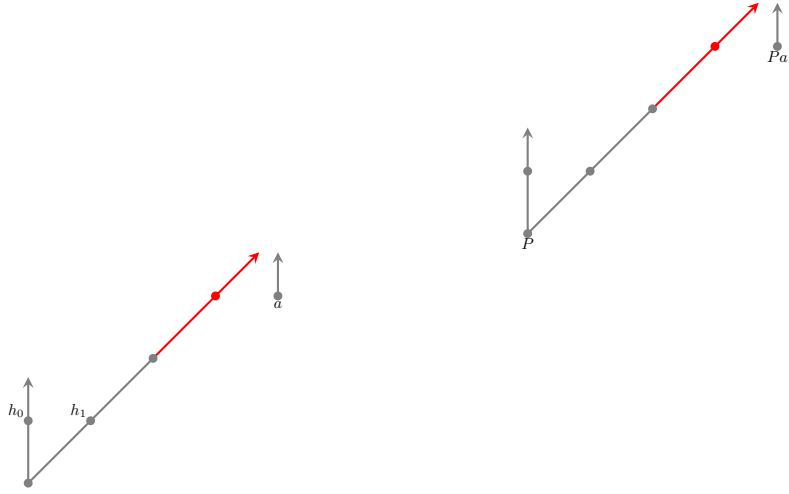

\begin{figure}
		\begin{tikzpicture}
		[scale=0.83,
		>=stealth,
		label distance=0,
		label position=below,
		every label/.style={
			inner sep=0,
			scale=0.60},
		every path/.style={thick},
		dot/.style={circle,
			inner sep=0,
			minimum size=0.09cm},
		tau0/.style={
			draw={gray},
			fill={gray}},
		tau1/.style={
			draw={red},
			fill={red}},
		tau2/.style={
			draw={darkgreen},
			fill={darkgreen}},
		tau3/.style={
			draw={amber},
			fill={amber}},
		tau4/.style={
			draw={purple},
			fill={purple}},
		tauextn/.style={draw={red}},
		tau2extn/.style={draw={orange}},
		tau3extn/.style={draw={darkgreen}},
		tau4extn/.style={draw={orange}},
		tau5extn/.style={draw={orange}},
		tau6extn/.style={draw={orange}},
		cl/.style={->},
		mot/.style={draw={darkcyan}},
		tower/.style={->},
		h0/.style={draw={gray}},
		h1/.style={draw={gray}},
		h2/.style={draw={gray}},
		d1/.style={draw={blue}},
		d2/.style={draw={darkcyan}},
		d3/.style={draw={red}},
		d4/.style={draw={darkgreen}},
		d5/.style={draw={blue}},
		d6/.style={draw={orange}},
		d7/.style={draw={orange}},
		d9/.style={draw={orange}},
		h0tower/.style={draw={gray}, tower},
		h1tower/.style={draw={gray}, tower},
        th0tower/.style={draw={red}, tower},
        th1tower/.style={draw={red}, tower},
        t2h0tower/.style={draw={darkgreen}, tower}, 
        t2h1tower/.style={draw={darkgreen}, tower},
        t3h0tower/.style={draw={amber}, tower}, 
        t3h1tower/.style={draw={amber}, tower}]

		\draw[h0, tau0] (0.0,0) -- (0.0,1);
		\draw[h1tower] (0.0,0) -- +(0.7,0.7);
		\node[dot, tau0, label=$ $] at (0.0,0) {};
		\draw[h0tower] (0.0,1) -- +(0.0,0.7);
		\node[dot, tau0, label=135:$ h_0$] at (0.0,1) {};
		\draw[h0, tau3extn] (4.0,2) -- (4.0,3);
		\draw[th1tower] (4.0,2) -- +(0.7,0.7);
		\node[dot, tau1, label=$ \overline{h_1^3}$] at (4.0,2) {};
		\draw[h0tower] (4.0,3) -- +(0.0,0.7);
		\node[dot, tau0, label=135:$ a$] at (4.0,3) {};
		\draw[h0, tau0] (8.0,4) -- (8.0,5);
		\draw[h1tower] (8.0,4) -- +(0.7,0.7);
		\node[dot, tau0, label=$ P$] at (8.0,4) {};
		\draw[h0tower] (8.0,5) -- +(0.0,0.7);
		\node[dot, tau0] at (8.0,5) {};
		\draw[h0, tau3extn] (12.0,6) -- (12.0,7);
		\draw[th1tower] (12.0,6) -- +(0.7,0.7);
		\node[dot, tau1, label=$ P \overline{h_{1}^3}$] at (12.0,6) {};
		\draw[h0tower] (12.0,7) -- +(0.0,0.7);
		\node[dot, tau0, label=135:$ P a$] at (12.0,7) {};
	\end{tikzpicture}
	\caption{$\text{Ext}_{\frac{\mathcal{A}(1)_*}{\tau}}$\label{komtchart}}
\end{figure}
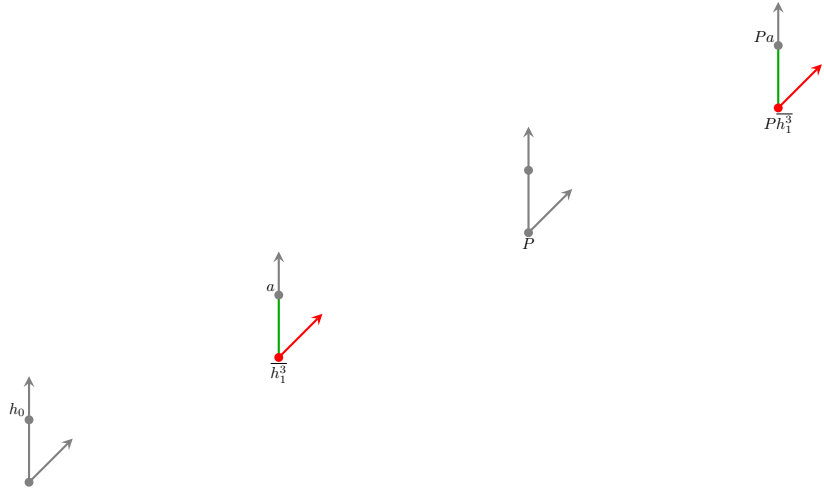

\newpage
\section{The Motivic May Spectral Sequence for \texorpdfstring{$\frac{\mathcal{A}(2)_*}{\tau}$}{Lg}}\label{sec4}

The $E_2$-page of the Adams spectral sequence for $\mathit{mmf}/\tau$ is isomorphic to $\text{Ext}_{\frac{\mathcal{A}(2)_*}{\tau}}$. In this section we compute $\text{Ext}_{\frac{\mathcal{A}(2)_*}{\tau}}$ as a ring using the May spectral sequence of Section \labelcref{MSS}.

\begin{prop}\label{4.1}
    The generators and $d_1$-differentials for the motivic May spectral sequence for $\frac{\mathcal{A}(2)_*}{\tau}$ are given in Table \ref{2}.
\end{prop}
\begin{proof}
Computing the coproduct on the indecomposables of $\frac{\mathcal{A}(2)_*}{\tau}$ shows that $\text{Gr} \frac{\mathcal{A}(2)_*}{\tau}$ is primitively generated. Since $\Ext$ over a primitively generated flat graded commutative Hopf algebra is polynomial on its generators \cite[Lemma 3.1.9]{Rav86}, we have 
\begin{align*}
    E_1 = \text{Ext}_{\text{Gr} \frac{\mathcal{A}(2)_*}{\tau}} \cong \F_2[h_0, h_1, h_2, h_{20}, h_{21}, h_{30}].
\end{align*}

The $d_1$ differentials are all derived from the reduced coproduct $\overline{\Delta}$ for the $\C$-motivic dual Steenrod algebra. We have 
\begin{align*}
    &\overline{\Delta}(\tau_0) = 0 &&\overline{\Delta}(\xi_1) = 0 && \overline{\Delta}(\xi_1^2) = 0 \\
    &\overline{\Delta}(\tau_1) = \xi_1 \otimes \tau_0 && \overline{\Delta}(\xi_2) = \xi_1^2 \otimes \xi_1 && \overline{\Delta}(\tau_2) = \xi_2 \otimes \tau_0 + \xi_1^2 \otimes \tau_1.
\end{align*}
The definition of the differentials in the cochain complex for the cobar complex of $\mathcal{A}_*$ gives
\begin{align*}
&d_1(h_{0}) = 0 &&d_1(h_{1}) = 0 &&d_1(h_{2}) = 0\\
    &d_1(h_{20}) = h_0 h_1 &&d_1(h_{21}) = h_1 h_2 &&d_1(h_{30}) = h_0 h_{21} + h_2 h_{20}.
\end{align*}
\end{proof}

\begin{prop}\label{4.2}
The $E_2$-page of the May spectral sequence for $\frac{\mathcal{A}(2)_*}{\tau}$ is given by the generators and relations in Table \ref{table:1} and Table \ref{table:2}. The $d_2$-differentials are given in the fourth column of Table \ref{table:1}. 
\end{prop}
\begin{proof}
The $E_2$-page of the May spectral sequence for $\frac{\mathcal{A}(2)_*}{\tau}$ is computed using the generators and $d_1$-differentials from Proposition \labelcref{4.1}. Recall from Section \labelcref{MSS} that reduction modulo $\tau$ induces a map of May spectral sequences, so the $d_2$-differentials in the May spectral sequence for $\frac{\mathcal{A}(2)_*}{\tau}$ must be compatible with the $d_2$-differentials in the motivic May spectral sequence for $\mathcal{A}(2)_*$. For example, we have 
\begin{align*}
    d_2(b_{20}) = \tau h_1^3 + h_0^2 h_2
\end{align*}
in the May spectral sequence for $\mathcal{A}(2)_*$ \cite{Isa09}. Therefore we have
\begin{align*}
    d_2(b_{20}) = h_0^2 h_2
\end{align*}
in the May spectral sequence for $\frac{\mathcal{A}(2)_*}{\tau}$.
\end{proof}

\begin{prop}\label{maygens}
$\text{Ext}_{\frac{\mathcal{A}(2)_*}{\tau}}$ has generators and relations given in Table \ref{3} and Table \ref{4}. 
\end{prop}

\begin{proof}
The $E_3$-page of the May spectral sequence for $\frac{\mathcal{A}(2)_*}{\tau}$ is computed using Proposition \labelcref{4.2}. For degree reasons there are no possible higher differentials and no possible hidden extensions. 
\end{proof}

\begin{remark}
    The elimination of higher differentials in the proof of Proposition \labelcref{maygens} is easier than the corresponding classical argument. The motivic weight spreads elements into distinct degrees and rules out many possibilities.
\end{remark}
\begin{remark}
    The notation for the generators 
    \begin{align*}
        &\overline{h_1^4} && \overline{h_1^2 c} &&\overline{u}
    \end{align*} is justified in Section \ref{5} where we show that these classes have the property
    \begin{align*}
        &q_*(\overline{h_1^4}) = h_1^4 && q_*(\overline{h_1^2 c}) = h_1^2 c && q_*(\overline{u}) = u
    \end{align*}
    where $q_*$ is the projection map from \labelcref{les}.
\end{remark}
We record some Massey products in $\text{Ext}_{\frac{\mathcal{A}(2)_*}{\tau}}$ which will be used in Section \labelcref{sec7}.
\begin{lemma}\label{mas}
We have the Massey products
    \begin{align*}
        & \overline{h_1^4} = \langle h_1, h_0, h_0 h_2 \rangle && c = \langle h_1, h_0, h_2^2 \rangle && u = \langle h_1, h_2, h_2^2 \rangle
    \end{align*}
    in $\text{Ext}_{\frac{\mathcal{A}(2)_*}{\tau}}$. There is no indeterminacy.
\end{lemma}
\begin{proof}
    We can compute these using the May $d_2$-differentials and the May convergence theorem \cite{May}. Proposition \labelcref{4.2} gives us May $d_2$-differentials 
    \begin{align*}
        &d_2(b_{20}) = h_0^2 h_2 &&d_2(h_0(1)) = h_0 h_2^2 &&d_2(b_{21}) = h_2^3.
    \end{align*} 
    Since $h_0 h_1 $ and $h_1 h_2$ are both killed by $d_1$-differentials, we have  
    \begin{align*}
        & h_1 b_{20} \subset \langle h_1, h_0, h_0 h_2 \rangle && h_1 h_0(1) \subset \langle h_1, h_0, h_2^2 \rangle && h_1 b_{21} \subset \langle h_1, h_2, h_2^2 \rangle
    \end{align*}
    where $\overline{h_1^4}$, $c$, and $u$ are detected by $h_1 b_{20}$, $h_1 h_0(1)$, and $h_1 b_{21}$ respectively. The lack of indeterminacy can be checked by inspection. For example, the element $h_0^2 h_2$ is hit by a unique May $d_2$-differential so there is a unique choice of nullhomotopy when constructing $\langle h_1, h_0, h_0 h_2 \rangle$.
\end{proof}

\section{Inclusion and Projection for \texorpdfstring{$\text{Ext}_{\frac{\mathcal{A}(2)_*}{\tau}}$}{Lg}}\label{5}

In this section we compute the values of the inclusion and projection maps in the long exact sequence

\[\begin{tikzcd}
	\cdots & {\text{Ext}_{\mathcal{A}(2)_*}^{*, *, *}} & {\text{Ext}_{\frac{\mathcal{A}(2)_*}{\tau}}^{*, *, *}} & {\text{Ext}_{\mathcal{A}(2)*}^{\mathord{*-1}, \mathord{*+1},\mathord{*+1}}} & \cdots
	\arrow["i_*", from=1-2, to=1-3]
	\arrow["q_*", from=1-3, to=1-4]
	\arrow[from=1-1, to=1-2]
	\arrow[from=1-4, to=1-5]
\end{tikzcd}\]
We record the values of inclusion and projection by renaming the elements of $\text{Ext}_{\frac{\mathcal{A}(2)_*}{\tau}}$ according to their behavior under $i_*$ and $q_*$.

\begin{prop}\label{5.1}
    The values of $i_*$ on the ring generators of $\text{Ext}_{\mathcal{A}(2)_*}$ are given in the third column of Table \ref{1}.
\end{prop}
\begin{proof}
The values of $i_*$ are computed by inspection using the exactness of the sequence \labelcref{les}. Specifically recall that the kernel of $i_*$ is precisely the $\tau$-divisible elements of $\text{Ext}_{\mathcal{A}(2)_*}$. The ring generators of $\text{Ext}_{\mathcal{A}(2)_*}$ are not divisible by $\tau$ and therefore have nontrivial image under $i_*$. The specific values of $i_*$ on the generators in Table \ref{1} are forced for degree reasons.
\end{proof}
 By Lemma \labelcref{iring} we have that $i_*$ is a ring homomorphism so Proposition \labelcref{5.1} totally determines the behavior of $i_*$ on $\text{Ext}_{\mathcal{A}(2)_*}$.
\begin{prop}
The indecomposables of $\text{Ext}_{\frac{\mathcal{A}(2)_*}{\tau}}$ as an $\text{Ext}_{\mathcal{A}(2)_*}$-module are given in Table \ref{12}.
\end{prop}
\begin{proof}
By inspection on tridegrees, none of the elements listed in Table \ref{12} are in the image of $i_*$. The $\text{Ext}_{\mathcal{A}(2)_*}$-module structure of  $\text{Ext}_{\frac{\mathcal{A}(2)_*}{\tau}}$ is determined by the map $i_*$. A product of the ring generators in Table \ref{3} is an $\text{Ext}_{\mathcal{A}(2)_*}$-module indecomposable if it is not divisible by an element in the image of $i_*$. The ring generators of $\text{Ext}_{\frac{\mathcal{A}(2)_*}{\tau}}$ which are not in the image of $i_*$ are precisely $\overline{h_1^4}$, $\overline{h_1^2 c}$, and $\overline{u}$ so we consider products consisting only of these three elements. Recall from Proposition \ref{5.1} that $i_*(\Delta^2) = \overline{u}^4$. Furthermore the relations in Table \ref{4} give us
\begin{align*}
    &i_*(h_1^2 P) = \overline{h_1^4} \mathord{\cdot} \overline{h_1^4} &&i_*(P c) = \overline{h_1^4} \mathord{\cdot} \overline{h_1^2 c} &&i_*(P d) = \overline{h_1^2 c} \mathord{\cdot} \overline{h_1^2 c}.
\end{align*}
Therefore the elements listed in Table \ref{12} are the only elements of $\text{Ext}_{\frac{\mathcal{A}(2)_*}{\tau}}$ which are not divisible by an element in the image of $i_*$.
\end{proof}

We can now justify the naming convention for the elements in our charts.

\begin{prop}\label{qmod}
The values of $q_*$ on the $\text{Ext}_{\mathcal{A}(2)_*}$-module indecomposables of $\text{Ext}_{\frac{\mathcal{A}(2)_*}{\tau}}$ are given in the third column of Table \ref{12}.
\end{prop}
\begin{proof}
The values of $q_*$ are computed by inspecting the long exact sequence \labelcref{les}. Specifically the map $q_*$ surjects onto the $\tau$-torsion elements of $\text{Ext}_{\mathcal{A}(2)_*}$. The values of $q_*$ on the $\text{Ext}_{\mathcal{A}(2)_*}$-module indecomposables of $\text{Ext}_{\frac{\mathcal{A}(2)_*}{\tau}}$ are then forced for degree reasons. For example we have 
\begin{align}\label{qring}
        &q_*(\overline{h_1^4}) = h_1^4 && q_*(\overline{h_1^2 c}) = h_1^2 c && q_*(\overline{u}) = u && q_*(\overline{u}^2) = \tau h_2 g.
    \end{align}

The values of $q_*$ on other module indecomposables are determined similarly. 
\end{proof}
\begin{remark}
    All values of $q_*$ are determined by the values of $q_*$ on the $\text{Ext}_{\mathcal{A}(2)_*}$-module indecomposables of $\text{Ext}_{\frac{\mathcal{A}(2)_*}{\tau}}$ and the ring structure of $\text{Ext}_{\mathcal{A}(2)_*}$ \cite{Isa09}. For example, since $q_*$ is an $\text{Ext}_{\mathcal{A}(2)_*}$-module map, we have $q_*(d \cdot \overline{u}^2) = d \cdot q_*(\overline{u}^2)$. Since $q_*(\overline{u}^2) = \tau h_2 g$ we have $d \cdot q_*(\overline{u}^2) = \tau h_2 g d$. Since $h_2 d = h_0 e$ in $\text{Ext}_{\mathcal{A}(2)_*}$ we conclude that $q_*(d \cdot \overline{u}^2) = \tau h_0 e g$. 
\end{remark}

\section{Motivic Adams differentials for \texorpdfstring{$\mathit{mmf}/\tau$}{Ag}}\label{sec6}

In this section we compute the differentials in the motivic Adams spectral sequence for $\mathit{mmf}/\tau$. The ring structure of $\text{Ext}_{\frac{\mathcal{A}(2)_*}{\tau}}$ is robust enough for all differentials to follow from a single differential via Leibniz propagation. There are several ways to derive the initial seed differential. We proceed by comparison with the motivic Adams spectral sequence for $\mathbb{S}/\tau$ which has been computed by machine up to the 110-stem \cite{aNSS}.

\begin{thm}
    The $d_2$-differentials on the ring generators of the Adams $E_2$-page for $\mathit{mmf}/\tau$ are given in Table \ref{3}. 
\end{thm}
\begin{proof}
The differentials on $h_0$, $h_1$, $h_2$, $\overline{h_1^4}$, and $P$ are all trivial for degree reasons. The differentials on $e$, $u$, $\overline{h_1^2 c}$, and $\overline{u}$ are determined below in Lemmas  \ref{lem6.2}, \ref{lem6.3}, \ref{6.3}, and \ref{6.4}. The differentials on $c$, $d$, and $g$ are determined by the differentials on $u$, $\overline{u}$, and $e$ along with the fact that $d_2 \circ d_2 = 0$.
\end{proof}
\begin{lemma}\label{lem6.2}
$d_2(e) = h_1^2 d$.
\end{lemma}
\begin{proof}
We have a unit map $\epsilon:\mathbb{S}/\tau \to \mathit{mmf}/\tau$ that induces a map on $\C$-motivic Adams spectral sequences. Recall the differential $$d_2(e_0) = h_1^2 d_0$$ in the $\C$-motivic Adams $E_2$-page for $\mathbb{S}/\tau$ (also known as the algebraic Novikov $E_2$-page) \cite{aNSS}. Since $e_0$ and $e$ have the same cobar representative in the May spectral sequence for $\mathcal{A}_*$ and $\mathcal{A}(2)_*$ respectively, we have 
\begin{align*}
    &\epsilon_*(h_1) = h_1 &&\epsilon_*(d_0) = d &&\epsilon_*(e_0) = e.
\end{align*} 
Since $\epsilon_*$ is a map of Adams spectral sequences, we have
\[d_2(e) = h_1^2 d.\qedhere\]
\end{proof}
\begin{remark}
    The proof of Lemma \labelcref{lem6.2} is not really homotopical and can be phrased in terms of $BP_*BP$-comodules. From this perspective, the $\C$-motivic Adams spectral sequence for $\mathbb{S}/\tau$ is the algebraic Novikov spectral sequence for $BP_*$. The algebraic Novikov spectral sequence for $BP_*$ can be computed by machine \cite{wang}.
\end{remark}
\begin{lemma}\label{lem6.3}
    $d_2(u) = h_1^2 c$.
\end{lemma}
\begin{proof}
    Recall from Table \ref{4} that $c u = h_1^2 e$. Then by Lemma \ref{lem6.2} and the Leibniz rule, we have $c \cdot d_2(u) = h_1^4 d$. Therefore $d_2(u)$ is nonzero so we must have $d_2(u) = h_1^2 c$ for degree reasons.
\end{proof}
\begin{proof}
    \textbf{Alt:} Recall from Table \ref{4} that $c u = h_1^2 e$. First we show that $d_2(c) = 0$. The only possible nonzero value of $d_2(c)$ is $h_1^2 \overline{h_1^4}$. If $d_2(c) = h_1^2 \overline{h_1^4}$ then we must have $d_2(u) = 0$ since $d_2 \circ d_2 = 0$. Lemma \ref{lem6.2} and the Leibniz rule then give us $h_1^4 d = h_1^2 \overline{h_1^4} \cdot u$. Recall from Table \ref{4} that $\overline{h_1^4} \cdot u = h_1^2 d + h_1^4 \overline{u}$. This gives us $h_1^4 d + h_1^6 \overline{u} = h_1^4 d$ which is a contradiction so $d_2(c) = 0$. Now by Lemma \ref{lem6.2} and the Leibniz rule, we have $c \cdot d_2(u) = h_1^4 d$. Therefore $d_2(u)$ is nonzero so we must have $d_2(u) = h_1^2 c$ for degree reasons.
\end{proof}
\begin{lemma}\label{6.3}
    $d_2(\overline{h_1^2 c}) = 0$.
\end{lemma}
\begin{proof}
By Lemma \ref{lem6.3} we have $d_2(c) = 0$ since $d_2 \circ d_2 = 0$. Therefore $d_2(c \cdot \overline{h_1^4}) = 0$ since $d_2(\overline{h_1^4}) = 0$ for degree reasons. The relations in Table \ref{4} give us $c \cdot \overline{h_1^4} = h_1^2 \cdot \overline{h_1^2 c}$ so we have $h_1^2 \cdot d_2(\overline{h_1^2 c}) = 0$. We can not have $d_2(\overline{h_1^2 c}) = P h_1^2$ since $P h_1^2 \cdot h_1^2 \neq 0$. Therefore we must have $d_2(\overline{h_1^2 c}) = 0$.
\end{proof}
\begin{lemma}\label{6.4}
    $d_2(\overline{u}) = \overline{h_1^2 c}$. 
\end{lemma}
\begin{proof}
Table \ref{4} gives us the relation $ cd = u \cdot \overline{h_1^2 c} + h_1^2 c \cdot \overline{u}$. By Lemma \ref{lem6.3} and the Leibniz rule we have 
\begin{align*}
     0 = h_1^2 c \cdot \overline{h_1^2 c} + h_1^2 c \cdot d_2(\overline{u}).
\end{align*}
Table \labelcref{4} gives us that $h_1^2 c \cdot \overline{h_1^2 c} \neq 0$ so we have $d_2(\overline{u}) = \overline{h_1^2 c}$.
\end{proof}

All higher differentials must be zero for degree reasons so the Adams spectral sequence for $\mathit{mmf}/\tau$ collapses at the $E_3$-page.

\section{The homotopy of \texorpdfstring{$\mathit{mmf}/\tau$}{Bg}}\label{sec7}

In this section we resolve some extensions in $\pi_{\ast, \ast} mmf/\tau$ that are not detected in the Adams $E_\infty$-page. We use these hidden extensions to give a full list of generators and relations for the bigraded homotopy ring $\pi_{\ast, \ast} mmf/\tau$. We then establish some hidden values of the inclusion and projection maps on homotopy.

\subsection{Ring Structure of $\pi_{\ast, \ast} \mathit{mmf}/\tau$}
See Definition 2.10 in \cite{IHES} for a precise definition of a hidden extension in the Adams spectral sequence. We say that a hidden $a$-extension from $b$ to $c$ is $D$-periodic if there is a hidden $a$-extension from $D^n b$ to $D^n c$ for all $n$. There are two notable extensions in $\pi_{\ast, \ast} \mathit{mmf}/\tau$ which are hidden on the Adams $E_\infty$-page for $\mathit{mmf}/\tau$. Namely there is a hidden $h_2$-extension from $h_2^2$ to $h_1 c$, and there is a hidden $h_1$-extension from $h_0^2 a$ to $P \overline{h_1^4}$. These are proved in Lemma \labelcref{lma1} and Lemma \labelcref{lma3} respectively. All hidden extensions in the Adams spectral sequence for $\mathit{mmf}/\tau$ are obtained from these two using periodicity. 
\begin{lemma}\label{lma1}
There is a hidden $h_2$-extension from $h_2^2$ to $h_1 c$ in $\pi_{*, *}\mathit{mmf}/\tau$.
\end{lemma}
\begin{proof}
Recall from Table \ref{3} that $d_2(u) = h_1^2 c$ and $u = \langle h_1, h_2, h_2^2 \rangle$ in $\text{Ext}_{\frac{\mathcal{A}(2)_*}{\tau}}$. Since $u$ supports an Adams differential, Moss's convergence theorem says that the bracket $\langle h_1, h_2, h_2^2 \rangle$ is not defined on $\pi_{\ast, \ast} \mathit{mmf}/\tau$ \cite{moss} \cite{Evamoss}. Therefore either $h_1 h_2 \neq 0$ or $h_2^3 \neq 0$ in $\pi_{\ast, \ast} \mathit{mmf}/\tau$. Certainly $h_1 h_2 = 0$ since $\pi_{4, 3} \mathit{mmf}/\tau = 0$ so we must have $h_2^3 \neq 0$. The only class in the appropriate bidegree is $h_1 c$.
\end{proof}
\begin{remark}
    In this section we use the notational scheme described in Notation \labelcref{barnotation}. In particular the classes $h_0 \mathord{\cdot}\overline{u}$, $\overline{u}^2$, and $\overline{u}^4$ are denoted by $a$, $\overline{\tau h_2 g}$, and $\Delta^2$ respectively.
\end{remark}
\begin{remark}
The hidden extension in Lemma \labelcref{lma1} is $g$-periodic and $\overline{\tau h_2 g}$-periodic.
\end{remark}
\begin{remark}
 Recall that in the $E_2$-page of the classical Adams spectral sequence for the sphere we have $h_2^3 + h_1^2 h_3 = 0$, but $\nu^3 + \eta^2 \sigma \neq 0$ in homotopy. The relation $\nu^3 + \eta^2 \sigma = \eta \epsilon$  was first proved by Toda using a homotopical argument with Toda brackets \cite{toda}. Analyzing the natural map $\mathbb{S} \to \mathit{mmf}/\tau$ provides a more elementary algebraic proof of this relation. The relation $\nu^3 + \eta^2 \sigma = \eta \epsilon$ follows immediately from Lemma \labelcref{lma1} and the fact that $\sigma = 0$ in $\pi_{\ast, \ast}\mathit{mmf}$. 
\end{remark}
\begin{lemma}\label{lma2}
    There is a Toda bracket $\langle h_0, h_0 h_2, P \rangle$ in $\pi_{\ast, \ast} \mathit{mmf}/\tau$ which is detected by $h_0^2 a$ in the Adams spectral sequence for $\mathit{mmf}/\tau$.
\end{lemma}
\begin{proof}
    Recall that we have a $d_2$-differential $d_2(h_0 a) = P h_0 h_2$ in the Adams spectral sequence for $\mathit{mmf}/\tau$. This gives us a Massey product $h_0^2 a = \langle h_0, h_0 h_2, P \rangle$ in the $E_3$-page of the Adams spectral sequence for $\mathit{mmf}/\tau$. Moss's convergence theorem gives us the Toda bracket $\langle h_0, h_0 h_2, P \rangle$ in $\pi_{\ast, \ast} \mathit{mmf}/\tau$ which is detected by $h_0^2 a$ in the Adams spectral sequence for $\mathit{mmf}/\tau$. This Toda bracket has indeterminacy given by $h_0^{2 + k} a$.
\end{proof}
\begin{lemma}\label{lma3}
There is a hidden $h_1$-extension from $h_0^2 a$ to $P \overline{h_1^4}$ in $\pi_{{\ast, \ast}} \mathit{mmf}/\tau$.
\end{lemma}
\begin{proof}
Lemma \labelcref{lma2} gives us that $h_0^2 a$ detects the Today bracket $\langle h_0, h_0 h_2, P \rangle$. We can shuffle brackets to get $h_1 \langle h_0, h_0 h_2, P \rangle = \langle h_1, h_0, h_0 h_2\rangle P$. Recall from Lemma \labelcref{mas} that $\overline{h_1^4} = \langle h_1, h_0, h_0 h_2 \rangle$ in $\text{Ext}_{\frac{\mathcal{A}(2)_*}{\tau}}$ which gives us an analogous Toda bracket in $\pi_{\ast, \ast} \mathit{mmf}/\tau$ by Moss's convergence theorem. 
\end{proof}
\begin{remark}
The hidden extension in Lemma \labelcref{lma3} is $P$-periodic and $\overline{\tau h_2 g}$-periodic.
\end{remark}

These hidden extensions let us compute a minimal set of generators and relations for $\pi_{\ast, \ast} \mathit{mmf}/\tau$ as a bigraded ring.

\begin{prop}
The bigraded homotopy ring $\pi_{\ast, \ast} \mathit{mmf}/\tau$ is given by the generators and relations in Table \ref{13} and Table \ref{14}.

\end{prop}
\begin{proof}
The ring structure of $\pi_{\ast, \ast} \mathit{mmf}/\tau$ can be determined by inspection using Lemma \labelcref{lma1}, Lemma \labelcref{lma2}, Lemma \labelcref{lma3}, and the Adams $E_\infty$-page for $\mathit{mmf}/\tau$.
\end{proof}

\subsection{Projection and Inclusion on Homotopy}\label{iqpi}
In this section we compute and analyze some values of the projection and inclusion maps 
\begin{align*}
    &i: \pi_{*, *} \mathit{mmf} \to \pi_{*, *} \mathit{mmf}/\tau &&q: \pi_{*, *} \mathit{mmf}/\tau \to \pi_{*-1, *+1} \mathit{mmf}
\end{align*}
on homotopy. Most of the values of these maps are determined by the computation in Section 5 of the maps
\begin{align*}
    &i_*: \text{Ext}_{\mathcal{A}(2)_*}^{*, *, *} \to \text{Ext}_{\frac{\mathcal{A}(2)_*}{\tau}}^{*, *, *} &&q_*: \text{Ext}_{\frac{\mathcal{A}(2)_*}{\tau}}^{*, *, *} \to \text{Ext}_{\mathcal{A}(2)_*}^{*-1, *+1, *+1}.
\end{align*}
on Adams $E_2$-pages. See Definition 2.6 in \cite{IHES} for a precise definition of a hidden value of a filtration preserving map.

\begin{prop}\label{hidq}
    Through the 60-stem, Table \labelcref{15} lists all hidden values of $q$.
\end{prop}
\begin{proof}
The hidden values of projection on homotopy are computed by inspection using the exactness of the sequence induced by $\tau$ on homotopy. Specifically recall that $q$ surjects onto the kernel of $\tau$. For example, inspecting the Adams $E_\infty$-page for $\mathit{mmf}$ shows that the class $P h_1^3$ in the Adams spectral sequence for $\mathit{mmf}$ detects $\tau$-torsion in the homotopy of $\mathit{mmf}$ \cite{Isa18}. So $P h_1^3$ is in the image of $q$. Inspecting the Adams $E_\infty$-page for $\mathit{mmf}/\tau$ shows that for degree reasons the only possible preimage is the class $h_0^2 a$ in the Adams spectral sequence for $\mathit{mmf}/\tau$. All hidden values of $q$ are determined similarly. 
\end{proof}
\begin{example}
Figure \labelcref{mmftles} displays a portion of the exact sequence 
\begin{align*}
    \cdots \to \pi_{24, 12} \mathit{mmf} \xrightarrow{i} \pi_{24, 12} \mathit{mmf / \tau} \xrightarrow{q} \pi_{23, 13} \mathit{mmf} \to \cdots
\end{align*}
of graded abelian groups. Analyzing this short exact sequence shows that $P h_1 d$ is a hidden value of $i$ with preimage $a^2$.
  \begin{figure}[ht]
  \centering
  \begin{tikzpicture}[scale=0.6]
    \draw ;
    \draw [blue, ->, thick] (0, 0) -- (1.85, 0);
    \draw [blue, ->, thick] (0, 1) -- (1.85, 1);
    \draw [blue, ->, thick] (2, -3) -- (3.85, -3);
    \draw [blue, ->, thick] (2, -2) -- (3.85, -2);
    \draw [purple, ->, thick] (2, -1) -- (3.85, -0.05);
    \node [left] at (0, 1) {$h_0^2 a^2$};
    \node [left] at (0, 0) {$h_0 a^2$};
    \draw [black, -, thick] (0, 0) -- (0, 1);
    \draw [fill, black] (0, 1) circle (0.1);
    \draw [fill, black] (0, 0) circle (0.1);
    \draw [black, ->, thick] (0, 1) -- (0, 2);

    \node [left] at (2, -3) {$\overline{\tau h_2 g}$};
    \node [left] at (2, -2) {$\overline{\tau h_0 h_2 g}$};
    \node [left] at (0, 0) {$h_0 a^2$};
    \draw [red, -, thick] (2, -3) -- (2, -2);
    \draw [darkgreen, -, thick] (2, -2) -- (2, -1);
    \draw [fill, red] (2, -3)  circle (0.1);
    \draw [fill, red] (2, -2)  circle (0.1);

    \node [left] at (2, -1) {$a^2$};
    \draw [black, fill] (2, -1)  circle (0.1);

    \node [right] at (2, 0) {$h_0 a^2$};
    \draw [black, fill] (2, 0)  circle (0.1);
    \draw [black, -, thick] (2, -1) -- (2, 0);
    \draw [black, -, thick] (2, 0) -- (2, 1);
    \draw [black, ->, thick] (2, 1) -- (2, 2);

    \node [right] at (2, 1) {$h_0^2 a^2$};
    \draw [black, fill] (2, 1)  circle (0.1);
    \draw [black, -, thick] (4, -3) -- (4, -2);
    
    \draw [black, -, thick, dotted] (4, -2) -- (4, 0);

    \node [right] at (4, -3) {$\tau h_2 g$};
    \draw [black, fill] (4, -3)  circle (0.1);

    \node [right] at (4, -2) {$\tau h_0 h_2 g$};
    \draw [black, fill] (4, -2)  circle (0.1);

    \node [right] at (4, 0) {$P h_1 d$};
    \draw [black, fill] (4, 0)  circle (0.1);

\end{tikzpicture}
  \caption{$\pi_{24, 12} \mathit{mmf} \to \pi_{24, 12} \mathit{mmf / \tau} \to \pi_{23, 13} \mathit{mmf}$}\label{mmftles}
\end{figure}
\end{example}
\begin{prop}\label{hidi}
    Through the 60-stem, the hidden values of $i$ are 
    \begin{align*}
        i(\tau^2 a g) = P \overline{\tau h_2 g} && i(\tau \Delta^2 c) = P \Delta^2.
    \end{align*}
\end{prop}

\begin{proof}
The hidden values of inclusion on homotopy are computed by inspection using the exactness of the sequence induced by $\tau$ on homotopy. Specifically recall that $i$ is injective on elements not divisible by $\tau$. For example, the class $\tau^2 a g$ in the Adams spectral sequence for $\mathit{mmf}$ detects an element of $\pi_{32, 16} \mathit{mmf}$ which is not divisible by $\tau$ \cite{Isa18}. Therefore $i(\tau^2 a g)$ is detected by a nontrivial class in the Adams spectral sequence for $\mathit{mmf}/\tau$. For degree reasons the only possible image is $P \overline{\tau h_2 g}$. Similarly, the class $\tau \Delta^2 c$ in the Adams spectral sequence for $\mathit{mmf}$ detects an element of $\pi_{56, 28} \mathit{mmf}$ which is not divisible by $\tau$ \cite{Isa18}. Therefore $i(\tau \Delta^2 c)$ must be detected by a nontrivial class in the Adams spectral sequence for $\mathit{mmf}/\tau$. For degree reasons the only possible image is $P \Delta^2$. 
\end{proof}

\subsection{Some Hidden Extensions in the Adams spectral sequence for $\mathit{mmf}$}
In this section we use the results of Section \labelcref{iqpi} to compute some hidden extensions in the Adams spectral sequence for $\mathit{mmf}$. The Adams spectral sequence for $\mathit{mmf}$ was computed in \cite{Isa18}. These hidden extensions are tabulated in Table \labelcref{17} and are not indicated on Isaksen's charts. Hidden extensions in the classical Adams spectral sequence for $\mathit{tmf}$ are completely tabulated in \cite{tmf}. The extensions in Table \labelcref{17} are exotic motivic phenomena in the sense that these indicate extensions between $\tau$-torsion elements in $\pi_{*, *} \mathit{mmf}$.

\begin{prop}\label{hidn}
Some hidden extensions in $\pi_{*, *}\mathit{mmf}$ are displayed in Table \labelcref{17}.
\end{prop}
\begin{proof}
The hidden extensions in Table \labelcref{17} can be proved using the results of Section \labelcref{iqpi} and the fact that $i$ and $q$ are both $\pi_{\ast, \ast} mmf$-module maps. For example, Proposition \labelcref{hidq} gives us that $q(a^2)$ is detected by $P h_1 d$ in the Adams spectral sequence for $\mathit{mmf}$. Proposition \labelcref{qmod} gives us that $q(\overline{\tau h_2 g})$ is detected by $\tau h_2 g$. Recall there is an $h_0$-extension from $\overline{\tau h_0 h_2 g}$ to $a^2$ in $\pi_{24, 12} \mathit{mmf}/\tau$. Since $q$ is a $\pi_{\ast, \ast} \mathit{mmf}$-module map, there is an $h_0$-extension from $\tau h_0 h_2 g$ to $P h_1 d$ in $\pi_{23, 13} \mathit{mmf}$. All hidden extensions in Table \ref{17} are determined similarly.
\end{proof}

\newpage

\section{Tables}

\begin{table}[hbt!]
    \centering
    \caption{Ring Generators for $\text{Ext}_{\mathcal{A}(2)_*}$}
    \begin{tabular}{l l l}\toprule
    $(s, f, w)$ & $x$ & $i_*(x)$  \\ [0.5ex] 
 \hline
    (0, 1, 0) & $h_0$ & $h_0$ \\
    (1, 1, 1) & $h_1$ & $h_1$ \\
    (3, 1, 2) & $h_2$ & $h_2$  \\
    (8, 4, 4) & $P$ &  $P$ \\
     (8, 3, 5) & $c$ & $c$\\
     (11, 3, 7) & $u$ & $u$ \\
     (12, 3, 6) & $a$ & $\overline{u} \cdot h_0$ \\
     (14, 4, 8) & $d$ & $d$ \\
     (15, 3, 8) & $n$ & $\overline{u} \cdot h_2$ \\
     (17, 4, 10) & $e$ & $e$ \\
     (20, 4, 12) & $g$ & $g$ \\
     (25, 5, 13) & $\Delta h_1$ & $\overline{u}^2 \cdot h_1$ \\
     (32, 7, 17) & $\Delta c$ & $\overline{u}^2 \cdot c$ \\
     (35, 7, 19) & $\Delta u$ & $\overline{u}^2 \cdot u$ \\
     (48, 8, 24) & $\Delta^2$ & $\overline{u}^4$ \\
     \bottomrule
    \end{tabular}
    \label{1}
\end{table}

\begin{table}[hbt!]
    \centering
    \caption{May $E_1$ generators for $\frac{\mathcal{A}(2)_*}{\tau}$}
    \begin{tabular}{l l l l}\toprule
    $(s, f, w)$ & May name & description & $d_1$ \\ [0.5ex] 
 \hline
    (0, 1, 0) & $h_0$ & $\tau_0$ & \\
    (1, 1, 1) & $h_1$ & $\xi_1$& \\
    (3, 1, 2) & $h_2$ & $\xi_1^2$ &  \\
    (2, 1, 1) & $h_{20}$ & $\tau_1$ &$h_0 h_1$ \\
     (5, 1, 3) & $h_{21}$ & $\xi_2$ & $h_1 h_2$\\
     (6, 1, 3) & $h_{30}$ & $\tau_2$ & $h_0 h_{21} + h_2 h_{20}$\\\bottomrule
    \end{tabular}
    \label{2}
\end{table}

\begin{table}[hbt!]
\centering
\caption{May $E_2$ generators for $\frac{\mathcal{A}(2)_*}{\tau}$}
\begin{tabular}{l l l l} 
 \toprule
 $(s, f, w)$ & May name & description & $d_2$ \\ [0.5ex] 
 \hline
  (0, 1, 0) & $h_0$ & $h_0$ &   \\ 
 (1, 1, 1) & $h_1$ & $h_1$ &  \\ 
 (3, 1, 2) & $h_2$ & $h_2$ &  \\ 
 (4, 2, 2) & $b_{20}$ & $h_{20}^2$ & $h_0^2 h_2$\\
 (7, 2, 4) & $h_0(1)$ & $h_{20}h_{21} + h_{1} h_{30}$ & $h_0 h_2^2$\\
 (10, 2, 6) & $b_{21}$ & $h_{21}^2$ & $h_2^3$\\
 (12, 2, 6) & $b_{30}$ & $h_{30}^2$ & \\

 \bottomrule
\end{tabular}
\label{table:1}
\end{table}

\begin{table}[hbt!]
\centering
\caption{May $E_2$ relations for $\frac{\mathcal{A}(2)_*}{\tau}$}
\begin{tabular}{l l} 
 \toprule
  $(s, f, w)$ & relation \\ [0.5ex] 
 \hline
   (1, 2, 1) & $h_0 h_1$\\ (4, 2, 3) & $h_1 h_2$\\(7, 3, 4) & $h_2 b_{20} + h_0 h_0(1)$ \\ (10, 3, 6) & $h_2 h_0(1) + h_0 b_{21}$ \\ (14, 4, 8) & $b_{20} b_{21} + h_0(1)^2 + h_1^2 b_{30}$ \\\bottomrule
\end{tabular}
\label{table:2}
\end{table}

\begin{table}[hbt!]
\label{tbl1}
\centering
\caption{$\text{Ext}_{\frac{\mathcal{A}(2)_*}{\tau}}$ generators}
\begin{tabular}{l l l l l} 
 \toprule
  $(s, f, w)$ & gen. & May name & Massey & $d_2$ \\ [0.5ex] 
 \hline
 (0, 1, 0) & $h_0$ & $h_0$ & & \\ 
 (1, 1, 1) & $h_1$ & $h_1$ & &\\ 
 (3, 1, 2) & $h_2$ & $h_2$ & & \\ 
 (5, 3, 3) & $\overline{h_1^4}$ & $h_1 b_{20}$ & $\langle h_1, h_0, h_0 h_2 \rangle$& \\
    (8, 3, 5) & $c$ & $h_1 h_0(1)$ & $\langle h_1, h_0, h_2^2 \rangle$& \\
    (8, 4, 4) & $P$ & $b_{20}^2$ & & \\
    (11, 3, 7) & $u$ & $h_1 b_{21}$ & $\langle h_1, h_2, h_2^2 \rangle$& $h_1^2 c$ \\
  (12, 2, 6) & $\overline{u}$ & $b_{30}$ & & $\overline{h_1^2 c}$\\
 (11, 4, 6) & $\overline{h_1^2 c}$ & $b_{20} h_0(1)$ & & \\
 (14, 4, 8) & $d$& $h_0(1)^2$ & & \\
 (17, 4, 10) & $e$ & $h_0(1) b_{21}$ && $h_1^2 d$\\
 (20, 4, 12) & $g$ & $b_{21}^2$ & 
 \\ [1ex] 
 \bottomrule
\end{tabular}
\label{3}
\end{table}

\newpage

\begin{table}[hbt!]
\centering
\caption{$\text{Ext}_{\frac{\mathcal{A}(2)_*}{\tau}}$ relations}
\begin{tabular}{l l} 
 \toprule
  $(s, f, w)$ & relation\\ [0.5ex] 
 \hline
    (1, 2, 1) & $h_0 h_1$ \\ (3, 3, 2) & $h_0^2 h_2$ \\(4, 2, 3) & $h_1 h_2$ \\
    (5, 4, 3) & $h_0 \cdot \overline{h_1^4}$ \\  (6, 3, 4) & $h_0 h_2^2$ \\ (8, 4, 5) & $h_2 \cdot \overline{h_1^4}$\\
    (8, 4, 5) & $h_0 c$\\ (9, 3, 6) & $h_2^3$\\
    (10, 6, 6) & $\overline{h_1^4} \cdot \overline{h_1^4} + P h_1^2$\\ (11, 4, 7) & $h_2 c$\\
    (11, 4, 7) & $h_0 u$ \\
    (11, 5, 6) & $h_0 \cdot \overline{h_1^2 c} + P h_2$\\
    (13, 6, 8) & $c \cdot \overline{h_1^4} + h_1^2 \cdot \overline{h_1^2 c}$\\ (14, 4, 9) & $h_2 u$\\ (14, 5, 8) & $h_2 \cdot \overline{h_1^2 c} + h_0 d$\\
    (16, 6, 10) & $u \mathord{\cdot} \overline{h_1^4} \mathord{+} h_1^2 d \mathord{+} h_1^4 \mathord{\cdot} \overline{u}$\\
    (16, 6, 10) & $c^2 + h_1^2 d$\\
    (16, 7, 9) & $\overline{h_1^4} \cdot \overline{h_1^2 c} + P c$\\
    (17, 5, 10) & $h_2 d + h_0 e$\\
    (19, 6, 12) & $cu + h_1^2 e$\\
    (19, 7, 11) & $Pu + h_1^2 \cdot \overline{h_1^4} \cdot \overline{u} + c \cdot \overline{h_1^2 c}$\\
    (19, 7, 11) & $Pu + d \cdot \overline{h_1^4} + h_1^2 \cdot \overline{u} \cdot \overline{h_1^4}$\\
    (20, 5, 12) & $h_2 e + h_0 g$\\
    (22, 6, 14) & $u^2 + h_1^2 g$\\
    (22, 7, 13) & $e \cdot \overline{h_1^4} + cd + h_1^2 c \cdot \overline{u}$\\
    (22, 7, 13) & $cd + h_1^2 c \cdot \overline{u} + u \cdot \overline{h_1^2 c}$\\
    (22, 8, 12) & $Pd + \overline{h_1^2 c} \cdot \overline{h_1^2 c}$\\
    (25, 7, 15) & $ce + h_1^2 u \cdot \overline{u} + g \cdot \overline{h_1^4}$\\
    (25, 7, 15) & $ce + ud$\\
    (25, 8, 14) & $Pe + d \cdot \overline{h_1^2 c} + h_1^2 \cdot \overline{u} \cdot \overline{h_1^2 c}$\\
    (28, 7, 17) & $cg + ue$\\
    (28, 8, 16) & $e \cdot \overline{h_1^2 c} + h_1^2 d \cdot \overline{u} + d^2$\\
    (28, 8, 16) & $Pg + d^2 + h_1^4 \cdot \overline{u}^2$\\
    (31, 8, 18) & $de + h_1^2 e \cdot \overline{u} + g \cdot \overline{h_1^2 c}$\\
    (34, 8, 20) & $e^2 + dg$\\
    \bottomrule
\end{tabular}
\label{4}
\end{table}
\phantom{love}
\newpage

\begin{table}[hbt!]
\centering
\caption{$\text{Ext}_{\mathcal{A}(2)_*}$-module indecomposables for $\text{Ext}_{\frac{\mathcal{A}(2)_*}{\tau}}$}
\begin{tabular}{l l l} 
 \toprule
 $(s, f, w)$ & $x$& $q_*(x)$ \\ [0.5ex] 
 \hline
 (0, 0, 0) & $1$ & $0$\\
  (5, 3, 3) & $\overline{h_1^4}$ & $h_1^4$\\
(11, 4, 6) & $\overline{h_1^2 c}$ & $h_1^2 c$\\
(12, 2, 6) & $\overline{u}$ & $u$\\
(17, 5, 9) & $\overline{u} \cdot \overline{h_1^4}$ & $h_1^2 d$\\
(23, 6, 12) & $\overline{u} \cdot \overline{h_1^2 c}$ & $cd$ \\
(24, 4, 11) & $\overline{u}^2$ & $\tau h_2 g$\\
(29, 7, 15) & $\overline{u}^2 \cdot \overline{h_1^4}$ & $\Delta h_1 \cdot h_1^3$\\
(35, 8, 18) & $\overline{u}^2 \cdot \overline{h_1^2 c}$ & $\Delta h_1 \cdot h_1c$\\
(36, 6, 18) & $\overline{u}^3$ & $\Delta u$\\
(41, 9, 21) & $\overline{u}^3 \cdot \overline{h_1^4}$ & $\Delta h_1 \cdot  h_1 d$\\
(47, 10, 24) & $\overline{u}^3 \cdot \overline{h_1^2 c}$ & $\Delta c \cdot d$\\
\bottomrule
\end{tabular}
\label{12}
\end{table}

\begin{table}[hbt!]
\caption{Generators for $\pi_{\ast, \ast} \mathit{mmf}/\tau$; or for the Adams-Novikov $E_2$-page for $\mathit{tmf}$}
\centering
\begin{tabular}{l l } 
 \toprule
 $(s, f, w)$ & gen. \\ [0.5ex] 
 \hline
 (0, 1, 0) & $h_0$ \\
 (1, 1, 1) & $h_1$ \\
 (3, 1, 2) & $h_2$  \\
 (5, 3, 3) & $\overline{h_1^4}$  \\
(8, 3, 5) & $c$  \\
 (8, 4, 4) & $P$  \\
 (12, 5, 6) & $h_0^2 a$  \\
 (14, 4, 8) & $d$ \\
 (20, 4, 12) & $g$  \\
 (24, 4, 12) & $\overline{\tau h_2 g}$ 
 \\ [1ex] 
 \bottomrule
\end{tabular}
\label{13}
\end{table}

\begin{table}[hbt!]
\centering
\caption{Relations for $\pi_{\ast, \ast} \mathit{mmf}/\tau$; or for the Adams-Novikov $E_2$-page for $\mathit{tmf}$}
\begin{tabular}{l l l} 
 \toprule
  $(s, f, w)$ & relation & proof \\ [0.5ex] 
 \hline
    (1, 2, 1) & $h_0 h_1$ \\ (3, 3, 2) & $h_0^2 h_2$ \\(4, 2, 3) & $h_1 h_2$ \\
    (5, 4, 3) & $h_0 \cdot \overline{h_1^4}$ \\  (6, 3, 4) & $h_0 h_2^2$ \\ (8, 4, 5) & $h_2 \cdot \overline{h_1^4}$\\
    (8, 3, 5) & $h_0 c$\\
    (11, 3, 7) & $h_2 c$\\
    (9, 3, 6) & $h_2^3 + h_1 c$ & Lemma \labelcref{lma1}\\
    (10, 6, 6) & $\overline{h_1^4} \cdot \overline{h_1^4} + P h_1^2$\\
    (11, 5, 6) & $h_2 P$\\
    (13, 6, 8) & $c \cdot \overline{h_1^4}$\\
    (13, 7, 7) & $h_1 \cdot h_0^2 a + P \cdot \overline{h_1^4}$ & Lemma \labelcref{lma3}\\
    (14, 5, 8) & $ h_0 d$\\ (14, 7, 9) & $h_2 \cdot h_0^2 a$ \\
    (16, 6, 10) & $h_1^2 d $\\
    (16, 6, 10) & $c^2$\\
    (16, 7, 9) & $P c$ \\ 
    (17, 8, 9) &  $\overline{h_1^4} \cdot h_0^2 a$\\
    (19, 7, 11) & $d \cdot \overline{h_1^4}$\\
    (20, 6, 12) &  $h_2^2 d + h_0^2 g$\\
    (20, 8, 11) &  $c \cdot h_0^2 a$\\
    (22, 7, 13) & $cd$\\
     (22, 8, 12) & $P d$\\
     (24, 10, 12) & $h_0^2 a \cdot h_0^2 a + h_0^6 \cdot \overline{\tau h_2 g}$ \\
      (26, 9, 14) &  $d \cdot h_0^2 a$\\
     (28, 8, 16) & $d^2$\\
      (28, 8, 16) & $P g + h_1^4 \cdot \overline{\tau h_2 g}$\\
      (31, 10, 19) & $h_0^2 a \cdot g$ \\

    \bottomrule
\end{tabular}
\label{14}
\end{table}

\phantom{love}
\newpage
\phantom{love}
\begin{table}[hbt!]
\centering
\caption{Some hidden values of $q$}
\begin{tabular}{l l l} 
 \toprule
   $(s, f, w)$ & source & target\\
 \hline
 (12, 5, 6) & $h_0^2 a$ & $P h_1^3$\\
 (24, 6, 12) & $a^2$ & $P h_1 d$ \\
(27, 6, 14) & $n^2$ & $h_1 d^2$ \\
 (36, 9, 18) & $a^3$ & $P \Delta h_1^3$  \\
(48, 8, 24) & $\Delta^2$ & $\tau ang$\\
(48, 9, 24) & $\Delta^2 h_0$ & $P \Delta h_1 d$\\
(49, 9, 25) & $\Delta^2 h_1$ & $\tau^2 d^2 g$\\
(56, 11, 29) & $\Delta^2 c$ & $Pang$\\
\bottomrule
\end{tabular}
\label{15}
\end{table}

\begin{table}[hbt!]
\centering
\caption{Some hidden values of $i$}
\begin{tabular}{l l l} 
 \toprule
   $(s, f, w)$ & source & target\\
 \hline
 (32, 8, 16) & $\tau^2 a g$ & $P \overline{\tau h_2 g}$\\
 (56, 11, 28) & $\tau \Delta^2 c$ & $P \Delta^2$\\
\bottomrule
\end{tabular}
\label{16}
\end{table}

\begin{table}[hbt!]
\centering
\caption{Some hidden extensions in the Adams spectral sequence for $mmf$}
\begin{tabular}{l l l l l} 
 \toprule
   $(s, f, w)$ & $h_i$ & source & target & proof\\
 \hline
 (23, 6, 12) & $h_0$ & $\tau h_0 h_2 g$ & $P h_1 d$ & $q(a^2)$\\
 (26, 6 , 16) & $h_2$ & $\tau h_2^2 g$ & $h_1 d^2$ & $q(n^2)$\\
(32, 8, 16) & $h_1$ & $\tau^2 a g$ & $P \Delta h_1$ & $i(\tau a g)$\\
(32, 10, 16) & $h_0$ & $h_0^2 \tau^2 a g$ & $P h_0 a^2$ & $i(\tau a g)$\\
(47, 10 , 25) & $h_0$ & $\tau a n g$ & $P \Delta h_1 d$ & $q(\Delta^2)$\\
(47, 10 , 25) & $h_1$ & $\tau a n g$ & $\tau^2 d^2 g$ & $q(\Delta^2)$\\
(56, 11, 28) & $h_0$ & $\tau \Delta^2 c$ & $\Delta^2 P h_0 + \tau^3 h_1 d e$ & $i(\tau \Delta^2 c)$\\
(57, 12, 29) & $h_1$ & $\tau \Delta^2 c$ & $\Delta^2 P h_1^2$ & $i(\tau \Delta^2 c)$\\
\bottomrule
\end{tabular}
\label{17}
\end{table}

\newpage

\section{Charts}
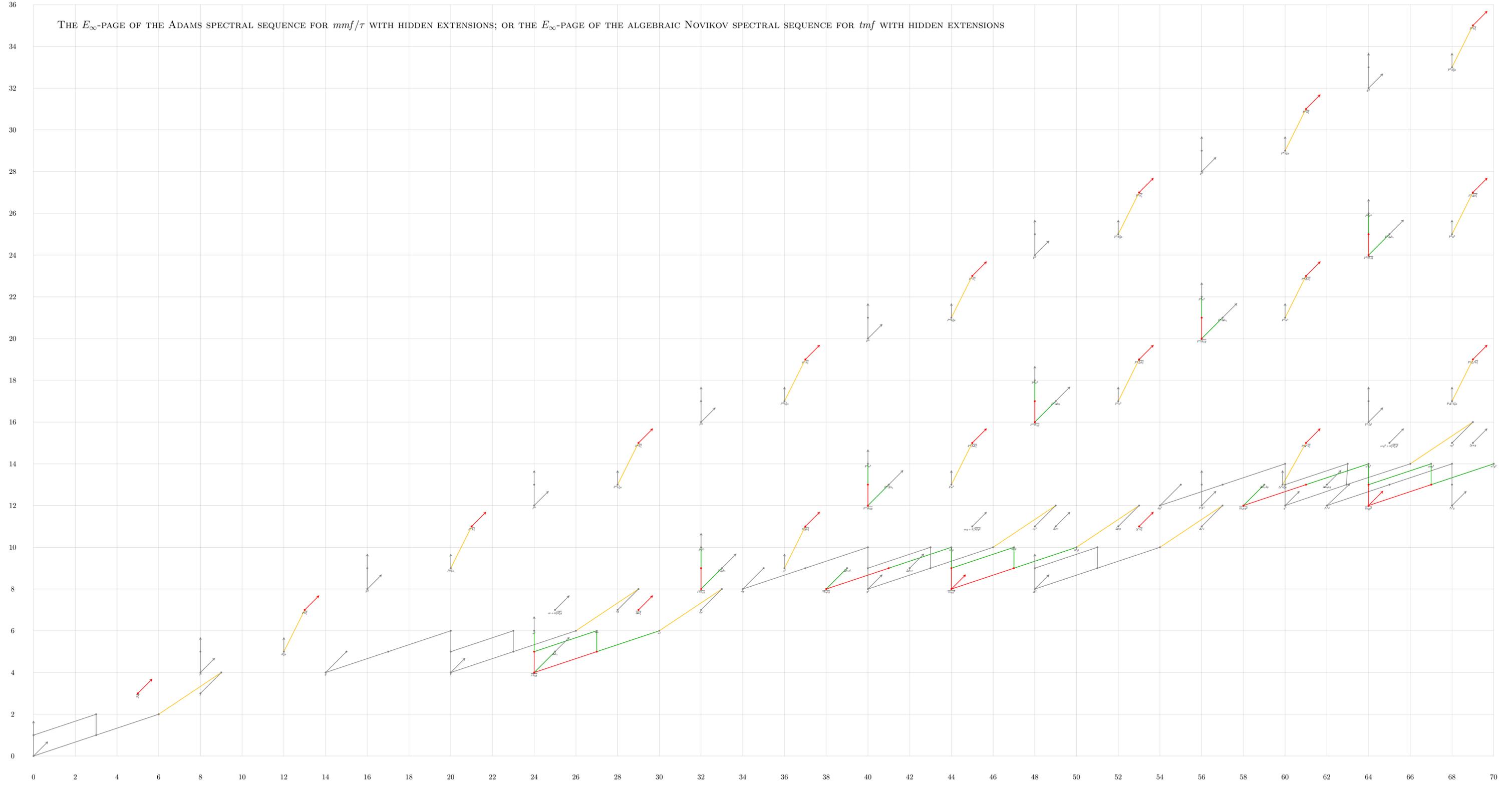
Chart 1 displays the $E_2$-page of the Adams spectral sequence for $\mathit{mmf}/\tau$, or the $E_2$-page of the algebraic Novikov spectral sequence for $\mathit{tmf}$, with differentials through the 70-stem. Chart 1 can be extended past the 70-stem by $P, g, \Delta^2$-periodicity. Chart 2 displays the $E_\infty$-page of the Adams spectral sequence for $\mathit{mmf}/\tau$, or the $E_\infty$-page of the algebraic Novikov spectral sequence for $\mathit{tmf}$, with hidden extensions through the 70-stem. Chart 2 can be extended past the 70-stem by $P, g, \Delta^2$-periodicity. For each fixed stem and filtration, the Adams spectral sequence for $\mathit{mmf}/\tau$ is a module over $\mathbb{F}_2$.

\begin{itemize}
    \item A solid gray dot \begin{tikzpicture}
        [dot/.style={circle,
				inner sep=0,
				minimum size=0.09cm,
				scale=1.10*1.5},tau0/.style={
				draw={chartgray},
				fill={chartgray}}]
    \node[dot, tau0, label=:$ $] at (0.0,0.0) {};
    \end{tikzpicture} indicates a copy of $\mathbb{F}_2$ in the associated graded object which is in the image of the inclusion $i_*$ of the bottom cell from \labelcref{les}.
    \item A solid red dot \begin{tikzpicture}
        [dot/.style={circle,
				inner sep=0,
				minimum size=0.09cm,
				scale=1.10*1.5},tau0/.style={
				draw={red},
				fill={red}}]
    \node[dot, tau0, label=:$ $] at (0.0,0.0) {};
    \end{tikzpicture} indicates a copy of $\mathbb{F}_2$ in the associated graded object which takes a nontrivial value under the projection $q_*$ to the top cell \labelcref{les}. 
     \item {\color{chartgray} Gray} lines of slope 0, 1, and 1/3 indicate multiplications by $h_0$, $h_1$, and $h_2$ respectively that are detected in the preimage of the map $i_*$. 
     \item {\color{red} Red} lines of slope 0, 1, and 1/3 indicate multiplications by $h_0$, $h_1$, and $h_2$ respectively that are detected in the image of the map $q_*$.
     \item {\color{darkgreen} Green} lines of slope 0, 1, and 1/3 indicate multiplications by $h_0$, $h_1$, and $h_2$ respectively that are hidden in the sense that their sources take nontrivial values under the map $q_*$ but their targets are in the image of the map $i_*$.
     \item Arrows of slope 1 indicate infinite towers of $h_1$ extensions.
     \item {\color{darkcyan} Teal} lines of negative slope indicate Adams differentials.
     \item {\color{amber} Yellow} lines indicate hidden extensions by $h_0$, $h_1$, and $h_2$.
\end{itemize}

\newpage

\newlength{\classpageheight}
\newlength{\classpagewidth}
\setlength{\classpageheight}{\pdfpageheight}
\setlength{\classpagewidth}{\pdfpagewidth}


\pdfpagewidth=80.0cm
\pdfpageheight=45.0cm

\input{mmftE2March2024} 
\input{mmftEooFeb2024}

\pdfpagewidth=\classpagewidth
\pdfpageheight=\classpageheight

\bibliographystyle{alpha} 
\bibliography{refs}
\end{document}